\documentclass[12pt]{amsart}
\usepackage{amsmath,amssymb,amsthm,amscd,mathtools}
\usepackage{graphicx}
\usepackage{booktabs}
\usepackage{enumitem}
\usepackage[hidelinks,pdfstartview=FitH]{hyperref}
\hypersetup{
  pdftitle={A proof of the framed LMOV conjecture for links},
  pdfauthor={Qingtao Chen, Kefeng Liu, Pan Peng, Shengmao Zhu},
  pdfkeywords={LMOV conjecture, framed HOMFLY-PT invariants, skein theory, Hecke algebras, quantum invariants}
}
\numberwithin{equation}{section}
\newtheorem{thm}{Theorem}[section]
\newtheorem{theorem}[thm]{Theorem}
\newtheorem{lemma}[thm]{Lemma}

\newtheorem{proposition}[thm]{Proposition}

\theoremstyle{definition}

\newtheorem{remark}[thm]{Remark}

\newcommand{\Z}{\mathbb Z}
\newcommand{\Q}{\mathbb Q}

\newcommand{\eps}{\varepsilon}

\newcommand{\ord}{\operatorname{ord}}

\newcommand{\Tr}{\operatorname{Tr}}

\newcommand{\Uq}[1]{U_{#1}(v)}
\newcommand{\HH}{\mathsf{H}}
\newcommand{\GG}{\mathsf{G}}

\begin{document}

\title[The framed LMOV conjecture]
{A proof of the framed LMOV conjecture for links}

\author[K. Liu]{Kefeng Liu}
\address{Mathematical Science Research Center\\Chongqing University of Technology\\Chongqing, China}
\email{liu@math.ucla.edu}

\author[P. Peng]{Pan Peng}

\address{Department of Mathematics, University of Arizona, 617 N. Santa Rita Ave., Tucson, AZ 85721, USA}
\email{pengpan@gmail.com}

\author[S. Zhu]{Shengmao Zhu}
\address{Department of Mathematics\\Zhejiang Normal University\\Jinhua, Zhejiang 321004, China}
\email{szhu@zju.edu.cn}
%\date{6 October 2026}

\begin{abstract}
We prove the Labastida--Mari\~no--Ooguri--Vafa conjecture for oriented links with arbitrary integer framings. Consequently, the framed reformulated invariants have finite integral LMOV expansions for every integer framing vector.
\end{abstract}

\maketitle

\section{Introduction}

The Labastida--Mari\~no--Ooguri--Vafa (LMOV) structure arose from the large-$N$ duality between Chern--Simons theory and topological strings; see the series of works \cite{Witten,OV,LMV,LM,MV}.  In mathematics, it predicts that the reformulated invariants, which are certain combinatorial expressions in colored HOMFLY--PT invariants, have only the universal pole $(q-q^{-1})^{-2}$ and admit an integral finite genus expansion.  Two of the authors, Liu and Peng treated the general zero-framed version in \cite{LP}.  The framed version and the related congruence-skein structures were studied in our previous work \cite{CLPZ}.  The purpose of this paper is to prove the framed LMOV conjecture for arbitrary integer framings, which includes the zero-framed theorem in \cite{LP} as a special case. We refer to Section 5 of \cite{CLPZ} for more background on the framed LMOV conjecture. Before stating the main results, let us fix the notation used throughout the paper.

\subsection{Notation and framing conventions}\label{subsec:notation}
We use the framed HOMFLY--PT skein in regular-isotopy normalization.  If $D_{\mathcal L}$ is a framed oriented diagram of a framed link $\mathcal L$, set
\[
\mathcal H(\mathcal L;q,a)=\langle D_{\mathcal L}\rangle,
\]
be the framed HOMFLY--PT polynomial of the link $\mathcal{L}$. It admits the skein relation
\begin{equation}\label{eq:skein-intro}
\mathcal H(\mathcal L_+)-\mathcal H(\mathcal L_-)=z\mathcal H(\mathcal L_0),\quad \text{and} \quad
\mathcal H(U)=\frac{a-a^{-1}}{z} \quad \text{for unknot $U$}.
\end{equation}
We write
\begin{equation}\label{eq:notation-qint}
 \{m\}=q^m-q^{-m},\qquad z=\{1\}=q-q^{-1},\qquad v=z^2,
 \qquad U_m(v)=\frac{\{m\}^2}{v}.
\end{equation}
Thus $U_m(v)\in\mathbb Z[v]$ is monic.  

For the annular Schur color $Q_A$ introduced in Section~\ref{sec:prelim}, a positive full ribbon twist has eigenvalue
\begin{equation}\label{eq:twist-eigen}
 a^{|A|}q^{\kappa_A},\qquad
 \kappa_A=\sum_iA_i(A_i-2i+1).
\end{equation}
For an oriented link $\mathcal L$ of $L$-components with any auxiliary component framing vector $\vec f=(f_1,...,f_L)$, we define the {\em framing-independent  invariant} by
\begin{equation}\label{eq:framing-independent-intro}
 W_{\vec A}(\mathcal L;q,a)
 :=q^{-\sum_b f_b\kappa_{A^b}}a^{-\sum_b f_b|A^b|}
 \mathcal H\!\left(\mathcal L_{\vec f}\star\bigotimes_{b=1}^LQ_{A^b};q,a\right).
\end{equation}
For an integer framing vector $\vec\tau=(\tau_1,\ldots,\tau_L)$, we define the {\em signed $q$-framed invariant} by 
\begin{equation}\label{eq:signed-frame-intro}
 W^{\vec\tau}_{\vec A}
 :=(-1)^{\sum_b\tau_b|A^b|}q^{\sum_b\tau_b\kappa_{A^b}}W_{\vec A}.
\end{equation}

For a partition $\mu$, put
\[
 \mathfrak z_\mu=\prod_{j\ge1}j^{m_j(\mu)}m_j(\mu)!,\qquad
 \{\mu\}=\prod_i\{\mu_i\}.
\]
A tuple of partitions is written with an arrow, for example
$\vec\lambda=(\lambda^1,\ldots,\lambda^L)$.  We write
$\ell(\vec\lambda)=\sum_c\ell(\lambda^c)$,
$\mathfrak z_{\vec\lambda}=\prod_c\mathfrak z_{\lambda^c}$,
$\{\vec\lambda\}=\prod_c\{\lambda^c\}$ and $p_{\vec\lambda}(x)=\prod_c p_{\lambda^c}(x^c)$. We also write
$\chi_{\vec A}(\vec\lambda)=\prod_c\chi_{A^c}(\lambda^c)$.  If $\vec\lambda$ is nonempty, let
$D_{\vec\lambda}$ be the greatest common divisor of all its nonzero parts.  When $k\mid D_{\vec\lambda}$, the notation $\vec\lambda/k$ means division of every part by $k$; similarly $k\vec\lambda$ means multiplication of every part by $k$.
The Adams operation is defined by
\[
 \Psi_k f(q,a)=f(q^k,a^k),\qquad \Psi_k\bigl(p_m(x^c)\bigr)=p_{km}(x^c).
\]
If $\ord(\alpha^2)=n$, then $\mathcal O_\alpha$ denotes the local coefficient ring at $q=\alpha$.

The {\em signed $q$-framed partition function} of a link $\mathcal{L}$ is defined by  
\begin{equation}\label{eq:intro-partition-function}
 Z^{\vec\tau}(\mathcal L;x)
 =\sum_{\vec A}W^{\vec\tau}_{\vec A}s_{\vec A}(x)
 =\sum_{\vec\lambda} Z^{\vec\tau}_{\vec\lambda}\frac{p_{\vec\lambda}(x)}{\mathfrak z_{\vec\lambda}},
\end{equation}
with $Z^{\vec\tau}_{\vec\varnothing}=1$, and the {\em signed $q$-framed free energy} is given by
\begin{equation}\label{eq:intro-free-energy}
 F^{\vec\tau}(\mathcal L;x)
 =\log Z^{\vec\tau}(\mathcal L;x)
 =\sum_{\vec\lambda\ne\vec\varnothing}F^{\vec\tau}_{\vec\lambda}
 \frac{p_{\vec\lambda}(x)}{\mathfrak z_{\vec\lambda}}.
\end{equation}

Set
\[
 \check F^{\vec\tau}_{\vec\lambda}
 :=\{\vec\lambda\}F^{\vec\tau}_{\vec\lambda},\qquad
 H^{\vec\tau}_{\vec\lambda}
 :=\frac{\{D\}^2\check F^{\vec\tau}_{\vec\lambda}}{\{\vec\lambda\}^2},
 \quad D=D_{\vec\lambda},\qquad r=\ell(\vec\lambda).
\]

The {\em signed $q$-framed reformulated power-sum coefficient invariant} and {\em Schur coefficient invariant} are given by 
\begin{equation}\label{eq:intro-key-normalization}
 C^{\vec\tau}_{\vec\lambda}
 =\frac{1}{\mathfrak z_{\vec\lambda}U_D(v)v}
 \sum_{k\mid D}\mu(k)k^{r-1}\Psi_k
 H^{\vec\tau}_{\vec\lambda/k}
\end{equation}
and 
\begin{equation}\label{eq:intro-schur}
    P_{\vec A}^{\vec\tau}=\sum_{\vec{\lambda}}\chi_{\vec A}(\vec \lambda)C^{\vec\tau}_{\vec\lambda}=\sum_{\vec{\lambda}}\frac{\chi_{\vec A}(\vec \lambda)}{\{\vec\lambda\}}\sum_{k\mid D}\frac{\mu(k)}k
 \Psi_k\left(
 \frac{F^{\vec\tau}_{\vec\lambda/k}}
 {\mathfrak z_{\vec\lambda/k}}
 \right)
\end{equation}
respectively. 
\begin{remark}
The framed LMOV conjecture, see Conjecture 5.2 in \cite{CLPZ}, states that $P_{\vec A}^{\vec\tau}\in v^{-1}\mathbb{Z}[v,a^{\pm 1}]$. Note that the notation $\mathcal F_{\vec{\lambda}}(\mathcal{L};q,a)$ in \cite{CLPZ} corresponds to $\frac{F^{\vec{\tau}}_{\vec\lambda}}{\mathfrak z_{\vec\lambda}}$ used here.  

\end{remark}

\subsection{Main results}
The first main result is the following congruence relation, which will be proved in Section \ref{sec:cable} based on the results established in Sections \ref{sec:hecke} and \ref{sec:exterior}.
\begin{theorem}\label{thm:framed-congruence}
Let $\mathcal L$ be an oriented link with integer framing vector $\vec\tau$, and let 
$\vec\nu$ be a nonempty partition tuple with total length $\ell(\vec\nu)=r$.  For every $n\ge2$ and every root $\alpha$ with $\ord(\alpha^2)=n$,
\begin{equation}\label{eq:framed-congruence}
 H^{\vec\tau}_{n\vec\nu}-n^{r-1}\Psi_nH^{\vec\tau}_{\vec\nu}
 \in(q-\alpha)^2\mathcal O_\alpha[a^{\pm1}].
\end{equation}
\end{theorem}

Using Theorem \ref{thm:framed-congruence} and Theorem \ref{thm:gcd} in Section \ref{sec:connected},  we prove $
 vC^{\vec\tau}_{\vec\lambda}\in\mathbb Q[v,a^{\pm1}]
$  in Section \ref{sec:rational}, 
and hence we have
\[
vP_{\vec A}^{\vec\tau}
=
\sum_{\vec\lambda}
\chi_{\vec A}(\vec\lambda)
vC_{\vec\lambda}^{\vec\tau}
\in
\mathbb Q[v,a^{\pm1}].
\]
On the other hand, using arithmetic methods, we obtain in Section \ref{sec:arithmetic},
\[
 P^{\vec\tau}_{\vec A}\in\Z[v,a^{\pm1},\{m\}^{-2}:m\ge1].
\]

Finally, combining the above results, we complete the proof of the framed LMOV conjecture. 
\begin{theorem}\label{thm:framed-LMOV}
Let $\mathcal L$ be an oriented link with integer framing vector $\vec\tau$.  For every nonempty tuple of partitions $\vec A$,
\begin{equation}\label{eq:intro-framed}
 vP^{\vec\tau}_{\vec A}(q,a)\in\mathbb Z[v,a^{\pm1}].
\end{equation}
Equivalently,
\begin{equation}\label{eq:intro-framed-expansion}
 P^{\vec\tau}_{\vec A}(q,a)
 =\sum_{g\ge0}\sum_{Q\in\mathbb Z}
 N^{\vec\tau}_{\vec A;g,Q}\,z^{2g-2}a^Q,
\end{equation}
where the coefficients are integers and only finitely many terms are nonzero for fixed $\mathcal L,\vec\tau,\vec A$.
\end{theorem}

\subsection{Related works}
The integrality conjecture formulated in~\cite{OV,LMV,LM} was referred to as the Labastida--Mari\~no--Ooguri--Vafa (LMOV) conjecture by Liu and Peng in~\cite{LP}. They obtained the zero-framed version of Theorem~\ref{thm:framed-LMOV} using the cabling technique introduced by Lin--Zheng~\cite{LinZheng}; see Proposition~\ref{prop:standard-LMOV} for a comparison between their formulation of ``Conjecture (LMOV)'' on p.~481 of~\cite{LP} and the zero-framed version of Theorem~\ref{thm:framed-LMOV} stated here. In 2012, building on the work of Mari\~no and Vafa~\cite{MV}, Liu and Peng~\cite{LP2} formulated the framed LMOV conjecture and proposed a proof framework using the HOMFLY--PT skein theory developed in~\cite{AM,MM,MortonMurphy,MortonPower}.
An essential missing ingredient in this framework is a collection of congruence skein relations for power-sum colored invariants.
 Some special congruence skein relations were subsequently studied by Chen–Liu–Peng–Zhu \cite{CLPZ}. Moreover, the Hecke lifting conjecture (Conjecture 5.3 in \cite{CLPZ}) is in fact a special case of a consequence of Theorem \ref{thm:framed-congruence} proved here. In \cite{LZ}, Luo–Zhu treated the framed LMOV integrality for the framed unknot using elementary number theory. More recently, Zhu \cite{Zhu} proved the framed LMOV conjecture for all torus links by expressing the reformulated invariants as contributions from connected combined graphs, via the vacuum expectation value and operator formalism developed in \cite{OP,Konishi}.

\subsection{AI Disclosure}
Building on an earlier 2012 note by Liu and Peng~\cite{LP2} which established the proof strategy, ChatGPT and GPT-6 Astra were used as auxiliary tools for literature searches, language editing, and checking mathematical constructions and arguments. Notably, GPT-6 Astra assisted in deriving the local crossing identities \eqref{eq:self-main} and \eqref{eq:mixed-main} from the congruence skein relations in~\cite{CLPZ}. The authors have independently verified all mathematical arguments and take full responsibility for the content of this paper.

%\textbf{Acknowledgements.}

\section{Partitions, skein colors, and the framed primitive transform}\label{sec:prelim}

\subsection{Partitions and symmetric functions}
A partition $\lambda=(\lambda_1,\lambda_2,\ldots)$ has weight $|\lambda|=\sum_i\lambda_i$ and length $\ell(\lambda)$.  Let $m_j(\lambda)$ be the multiplicity of the part $j$ and put
\begin{equation}\label{eq:zlambda}
 \mathfrak z_\lambda=\prod_{j\ge1}j^{m_j(\lambda)}m_j(\lambda)!,
 \qquad
 \kappa_\lambda=\sum_i\lambda_i(\lambda_i-2i+1).
\end{equation}
For a partition $\mu$ define
\begin{equation}\label{eq:quantum-mu}
 \{\mu\}=\prod_{i=1}^{\ell(\mu)}\{\mu_i\}.
\end{equation}
For partitions $A$ and $\lambda$, let $p_{\lambda}(x)$ and $s_A(x)$ be the power sum symmetric function and Schur symmetric function respectively, they are related by the Frobenius formulas:
\begin{equation}\label{eq:frob}
 s_A(x)=\sum_{\lambda\vdash|A|}\frac{\chi_A(\lambda)}{\mathfrak z_\lambda}p_\lambda(x),
 \qquad
 p_\lambda(x)=\sum_{A\vdash|\lambda|}\chi_A(\lambda)s_A(x).
\end{equation}

A partition tuple is always written with an arrow:
$\vec\lambda=(\lambda^1,\ldots,\lambda^L)$.  We use componentwise conventions:
\[
 \mathfrak z_{\vec\lambda}=\prod_c\mathfrak z_{\lambda^c}, \
 \ell(\vec\lambda)=\sum_c\ell(\lambda^c),\ 
 \{\vec\lambda\}=\prod_c\{\lambda^c\},\ 
 p_{\vec\lambda}(x)=\prod_c p_{\lambda^c}(x^c), \  s_{\vec A}(x)=\prod_c s_{A^c}(x^c).
\]
We also write
$\chi_{\vec A}(\vec\lambda)=\prod_c\chi_{A^c}(\lambda^c)$.
For a tuple of partitions $\vec A$, we also use
$|\vec A|=\sum_c|A^c|$ and
$\vec A^t=((A^1)^t,\ldots,(A^L)^t)$.
For a nonempty $\vec\lambda$, let $D_{\vec\lambda}$ denote the greatest common divisor of all nonzero parts.  If $k\mid D_{\vec\lambda}$, then $\vec\lambda/k$ means that every part of every component partition is divided by $k$.

\subsection{The HOMFLY--PT skein of the annulus}
Let $\mathcal C$ be the framed HOMFLY--PT skein of the annulus over the usual localized coefficient ring.  Its positive subalgebra $\mathcal C_+$ is canonically identified with the ring of symmetric functions.  Let $Q_A\in\mathcal C_+$ be the annular closure of the type-$A$ Hecke idempotent corresponding to the partition $A$; under the symmetric-function identification, $Q_A$ corresponds to $s_A$ \cite{AM,Lukac}.  Likewise $p_m$ corresponds to the annular power-sum color $P_m$ \cite{MortonPower,MM}.  Thus
\begin{equation}\label{eq:Pmu}
 P_\mu=\prod_iP_{\mu_i}=\sum_A\chi_A(\mu)Q_A,
 \qquad
 Q_A=\sum_\mu\frac{\chi_A(\mu)}{\mathfrak z_\mu}P_\mu.
\end{equation}
The polynomial cycle color is
\begin{equation}\label{eq:Bm}
 B_m=\{m\}P_m.
\end{equation}
Let $\HH_m$ be the framed HOMFLY--PT skein of $m$-tangles, i.e. the Hecke algebra over the usual localized coefficient ring. An integral Newton representative $b_m\in\HH_m$ is
\begin{equation}\label{eq:newton-coupon}
 b_m=z\sum_{j=0}^{m-1}g_1\cdots g_{m-1-j}
 g^{-1}_{m-j}\cdots g^{-1}_{m-1},
\end{equation}
with the evident convention for empty products.  Its annular closure is $B_m$, and $b_m$ lies in the integral Hecke algebra because $g_i^{-1}=g_i-z$.

\subsection{Framing-independent and integer-framed colored invariants}
Let $\mathcal L=K_1\cup\cdots\cup K_L$ be an oriented link and let $\vec f=(f_1,\ldots,f_L)$ be any auxiliary framing vector. Write $\mathcal{L}_{\vec f}$ for the resulting framed link and set
\[
 Q_{\vec A}:=\bigotimes_{b=1}^LQ_{A^b}.
\]
Decorating the $b$th framed tubular neighborhood by $Q_{A^b}$ gives the framed satellite $\mathcal L_{\vec f}\star Q_{\vec A}$.  We define the framing-independent colored HOMFLY--PT invariant by
\begin{equation}\label{eq:framing-independent-colored}
 W_{\vec A}(\mathcal L;q,a)
 :=q^{-\sum_b f_b\kappa_{A^b}}
 a^{-\sum_b f_b|A^b|}
 \mathcal H(\mathcal L_{\vec f}\star Q_{\vec A};q,a).
\end{equation}
The twist eigenvalue \eqref{eq:twist-eigen} shows that the right-hand side is independent of $\vec f$.  
Note that no division by the value of the colored unknot is performed; so $W_{\vec A}$ is the usual unreduced colored HOMFLY--PT invariant, normalized to be framing independent.

For an integer framing vector $\vec\tau$, our {\em signed $q$-framing invariant} is
\begin{equation}\label{eq:signed-frame}
 W^{\vec\tau}_{\vec A}
 =(-1)^{\sum_b\tau_b|A^b|}
 q^{\sum_b\tau_b\kappa_{A^b}}W_{\vec A},
\end{equation} 

For a partition tuple $\vec\mu$, the {\em signed $q$-framed moment} is defined by character transform,
\begin{equation}\label{eq:Zmu}
 Z^{\vec\tau}_{\vec\mu}
 =\sum_{\vec A}\chi_{\vec A}(\vec\mu)W^{\vec\tau}_{\vec A},
 \qquad
 \check Z^{\vec\tau}_{\vec\mu}=\{\vec\mu\}Z^{\vec\tau}_{\vec\mu}.
\end{equation}

\begin{lemma}\label{lem:framed-parallel}
Suppose the parts \(m_1,\ldots,m_s\) belong to one original component with integer framing \(\tau\). The corresponding signed $q$-framed moment is
\[
(-a^{-1})^{\tau(m_1+\cdots+m_s)}
\]
times the ordinary framed skein evaluation of the \(s\) parallel curves in the \(\tau\)-framed tubular neighborhood.
\end{lemma}
\begin{proof}
Put \(\mu=(m_1,\ldots,m_s)\) and \(M=|\mu|=m_1+\cdots+m_s\), and assume that these parts belong to the first component \(K_1\). By the twist eigenvalue and the definition of the signed \(q\)-framing,
\[
\begin{aligned}
Z^\tau_\mu
&=
\sum_{A\vdash M}\chi_A(\mu)
(-a^{-1})^{\tau M}
\mathcal H\!\left(K_1^\tau\star Q_A\cup K_2\cup\cdots\cup K_L;q,a\right)\\
&=
(-a^{-1})^{\tau M}
\mathcal H\!\left(K_1^\tau\star P_\mu\cup K_2\cup\cdots\cup K_L;q,a\right),
\end{aligned}
\]
where
\[
P_\mu=P_{m_1}\cdots P_{m_s}
=\sum_{A\vdash M}\chi_A(\mu)Q_A.
\]
Thus the cycle moment is evaluated by the \(s\) parallel cycle colors in the \(\tau\)-framed tubular neighborhood, multiplied by
\[
(-a^{-1})^{\tau M}.
\]
\end{proof}

\subsection{Partition function and reformulated invariants}
Define the partition function of the signed $q$-framed invariant as
\begin{equation}\label{eq:partition-function}
 Z^{\vec\tau}(x)
 =\sum_{\vec A}W^{\vec\tau}_{\vec A}s_{\vec A}(x)
 =\sum_{\vec\lambda} Z^{\vec\tau}_{\vec\lambda}\frac{p_{\vec\lambda}(x)}{\mathfrak z_{\vec\lambda}},
\end{equation}
with $Z^{\vec\tau}_{\vec\varnothing}=1$, and the free energy 
\begin{equation}\label{eq:free-energy}
 F^{\vec\tau}(x)
 =\log Z^{\vec\tau}(\mathcal L;x)
 =\sum_{\vec\lambda\ne\vec\varnothing}F^{\vec\tau}_{\vec\lambda}
 \frac{p_{\vec\lambda}(x)}{\mathfrak z_{\vec\lambda}}.
\end{equation}
Put
\begin{equation}\label{eq:checkFlambda}
 \check F^{\vec\tau}_{\vec\lambda}
 :=\{\vec\lambda\}F^{\vec\tau}_{\vec\lambda}.
\end{equation}
If the $r=\ell(\vec\lambda)$ cycle parts are temporarily labeled, {\em moment--cumulant inversion} is
\begin{equation}\label{eq:cumulant}
 \check F^{\vec\tau}_{\vec\lambda}
 =\sum_{\pi\in\Pi_r}(-1)^{|\pi|-1}(|\pi|-1)!
 \prod_{B\in\pi}\check Z^{\vec\tau}_B.
\end{equation}
Here $\Pi_r$ denotes the set of set partitions of the temporarily
labeled cycle parts $\{1,\ldots,r\}$.  If
\[
\pi=\{B_1,\ldots,B_k\}\in\Pi_r,
\]
then $|\pi|=k$, and for a block $B\subset\{1,\ldots,r\}$,
$\check Z^{\vec\tau}_B$ denotes the moment obtained by retaining
exactly the cycle parts whose labels belong to $B$, with their
original link-component assignments unchanged.

For example, suppose first that $\mathcal L=K$ is a knot and
$\lambda=(m,n)$.  Label the two cycle parts by $1$ and $2$.  Since
\[
\Pi_2=\bigl\{\{1,2\},\,\{1\}\{2\}\bigr\},
\]
formula~\eqref{eq:cumulant} becomes
\[
\check F^\tau_{(m,n)}
=
\check Z^\tau_{(m,n)}
-
\check Z^\tau_{(m)}
\check Z^\tau_{(n)}.
\]
As a second example, let $\mathcal L=K_1\cup K_2$ and take
\[
\vec\lambda=((m),(n)).
\]
Label the cycle on $K_1$ by $1$ and the cycle on $K_2$ by $2$.
Then
\[
\check F^{\vec\tau}_{((m),(n))}
=
\check Z^{\vec\tau}_{((m),(n))}
-
\check Z^{\vec\tau}_{((m),\varnothing)}
\check Z^{\vec\tau}_{(\varnothing,(n))}.
\]
Thus $\check F^{\vec\tau}_{((m),(n))}$ measures precisely the
connected contribution between the two labeled cycle colors; in
particular it vanishes when the two labeled groups split.

Let $Y=(Y^1,\ldots,Y^L)$ be the divided alphabet
\begin{equation}\label{eq:divided-alphabet}
 p_m(Y^c)=\frac{p_m(y^c)}{\{m\}}.
\end{equation}
Then, by \eqref{eq:free-energy} we have
\[
\log Z^{\vec{\tau}}[Y]=\sum_{\vec\lambda\ne\vec\varnothing}F^{\vec\tau}_{\vec\lambda}
 \frac{p_{\vec\lambda}(Y)}{\mathfrak z_{\vec\lambda}}=\sum_{\vec\lambda\ne\vec\varnothing}F^{\vec\tau}_{\vec\lambda}
 \frac{p_{\vec\lambda}(y)}{\mathfrak z_{\vec\lambda}\{\vec \lambda\}}
\]

Taking the plethystic logarithm, we write
\begin{equation}\label{eq:PLog}
 \operatorname{PLog}Z^{\vec\tau}[Y]
 =\sum_{k\ge1}\frac{\mu(k)}k\Psi_k\log Z^{\vec\tau}[Y]
 =\sum_{\vec\lambda} C^{\vec\tau}_{\vec\lambda} p_{\vec\lambda}(y)
 =\sum_{\vec A}P^{\vec\tau}_{\vec A}s_{\vec A}(y).
\end{equation}
Let $D=D_{\vec\lambda}$, then
\begin{equation}\label{eq:Clambda}
 C^{\vec\tau}_{\vec\lambda}
 =\sum_{k\mid D}\frac{\mu(k)}k
 \Psi_k\left(
 \frac{F^{\vec\tau}_{\vec\lambda/k}}
 {\mathfrak z_{\vec\lambda/k}\{\vec\lambda/k\}}
 \right)
\end{equation}
and 
\begin{equation}\label{eq:PA}
    P_{\vec A}^{\vec\tau}=\sum_{\vec{\lambda}}\chi_{\vec A}(\vec \lambda)C^{\vec\tau}_{\vec\lambda}, 
\end{equation}
are referred to as the reformulated power-sum coefficient invariant and Schur  coefficient invariant, respectively.

Put
\begin{equation}\label{eq:Hlambda-def}
 H^{\vec\tau}_{\vec\lambda}
 =\frac{\{D\}^2\check F^{\vec\tau}_{\vec\lambda}}{\{\vec\lambda\}^2},
 \qquad D=D_{\vec\lambda}.
\end{equation}
Using $\mathfrak z_{\vec\lambda}=k^r\mathfrak z_{\vec\lambda/k}$ and $\Psi_k\{D/k\}=\{D\}$ gives
\begin{equation}\label{eq:key-normalization}
 vC^{\vec\tau}_{\vec\lambda}
 =\frac{1}{\mathfrak z_{\vec\lambda}\,\Uq{D}}
 \sum_{k\mid D}\mu(k)k^{r-1}\Psi_k
 H^{\vec\tau}_{\vec\lambda/k}.
\end{equation}

\begin{proposition}\label{prop:standard-LMOV}
Let $f^{\vec\tau}_{\vec A}(q,a)$ be the reformulated colored invariants determined by
\begin{equation}\label{eq:standard-reformulated}
 \log Z^{\vec\tau}(x)
 =\sum_{d\ge1}\frac1d\sum_{\vec A\ne\vec\varnothing}
 f^{\vec\tau}_{\vec A}(q^d,a^d)
 \prod_{c=1}^{L}s_{A^c}\bigl((x^c)^d\bigr).
\end{equation}
For partitions $A,B$ of the same size define, in the present normalization,
\begin{equation}\label{eq:MAB}
 M_{AB}(q)=\sum_{\mu\vdash |A|}
 \frac{\chi_A(\mu)\chi_B(\mu)}{\mathfrak z_\mu}
 \prod_{i=1}^{\ell(\mu)}\{\mu_i\}.
\end{equation}
Then
\begin{equation}\label{eq:fMP}
 f^{\vec\tau}_{\vec A}(q,a)
 =\sum_{\vec B}P^{\vec\tau}_{\vec B}(q,a)
 \prod_{c=1}^{L}M_{A^cB^c}(q).
\end{equation}
Consequently, for $\vec\tau=0$, the coefficients $P^0_{\vec B}$ are exactly the standard LMOV reformulated invariants of \cite{LP} under the variable convention $q_{\mathrm{LP}}=q^{-2}$ and $t_{\mathrm{LP}}=a^{-2}$.
\end{proposition}
\begin{proof}
The M\"obius operator in \eqref{eq:PLog} is inverse to the multiple-cover expansion \eqref{eq:standard-reformulated}.  Hence, before dividing the alphabet,
\[
 \operatorname{PLog}Z^{\vec\tau}(x)
 =\sum_{\vec A\ne\vec\varnothing}f^{\vec\tau}_{\vec A}(q,a)s_{\vec A}(x).
\]
By Frobenius' formula and \eqref{eq:divided-alphabet}, componentwise,
\[
 s_A[Y]
 =\sum_{\mu\vdash |A|}
 \frac{\chi_A(\mu)}{\mathfrak z_\mu\{\mu\}}p_\mu(y)
 =\sum_B (M(q)^{-1})_{AB}s_B(y).
\]
Indeed the character table is orthogonal with weight $\mathfrak z_\mu^{-1}$, so the matrices obtained by weighting the $\mu$-column by $\{\mu\}$ and by $\{\mu\}^{-1}$ are inverse.  Therefore the Schur coefficient invariant after the divided-alphabet substitution is $M(q)^{-1}f$, or equivalently $f=M(q)P$.  Taking the tensor product over the link components gives \eqref{eq:fMP}.  This is the standard LMOV change of basis in \cite{LP}.  More explicitly, with $q_{\mathrm{LP}}=q^{-2}$ one has
\[
 q_{\mathrm{LP}}^{-m/2}-q_{\mathrm{LP}}^{m/2}=q^m-q^{-m}=\{m\},
\]
so the matrix \eqref{eq:MAB} is exactly the Liu--Peng matrix \cite{LP} in the present variables (with $t_{\mathrm{LP}}=a^{-2}$).
\end{proof}

\section{Connected cumulants estimates and the gcd denominator}\label{sec:connected}

The purpose of this section is to obtain the sharp denominator estimate
for the connected cycle coefficients. The estimates will first be proved for ordinary framed skein
evaluations.  Lemma~\ref{lem:framed-parallel} then allows us to apply them to the framed
parallel realization of the signed cycle moments
$\check Z_{\vec\lambda}^{\vec\tau}$ and hence to the connected
coefficients $\check F_{\vec\lambda}^{\vec\tau}$.

\subsection{Connected cumulants estimates}
Let the labels of a link $\mathcal{L}$  be indexed by $\{1,\ldots,r\}$.
For every subset
$
T\subseteq[r],
$
let
$
M(T)
$
denote the moment obtained by retaining precisely the labels in $T$.
Then the connected cumulant is
\[
\operatorname{Cum}(\mathcal{L})
=
\sum_{\pi\in\Pi_r}
(-1)^{|\pi|-1}(|\pi|-1)!
\prod_{B\in\pi}M(B),
\]
which is exactly the labeled moment--cumulant formula~\eqref{eq:cumulant} if $M(B)=\check{Z}_{B}^{\vec{\tau}}$..

The two elementary properties below will be used repeatedly.

\begin{lemma}
\label{lem:cumulant-contraction}
Let $\mathcal L$ be a finite labeled framed link.

\medskip
\noindent
{\rm (i)}
If the labels can be divided into two nonempty groups which are
separated by an embedded $2$-sphere, then
\[
\operatorname{Cum}(\mathcal L)=0.
\]

\medskip
\noindent
{\rm (ii)}
Suppose a local change involves two labels $i$ and $j$, and assume
that for every sublink retaining both labels the corresponding moment
changes by
\[
\Delta M(T)
=
c\,M(T_{i\sim j}),
\qquad
i,j\in T,
\]
where $T_{i\sim j}$ denotes the label set obtained by merging $i$ and
$j$ into a single label.  Assume also that
$
\Delta M(T)=0
$
whenever $T$ omits $i$ or $j$.  Then
\[
\Delta\operatorname{Cum}(\mathcal L)
=
c\,\operatorname{Cum}(\mathcal L_{i\sim j}).
\]
The same conclusion holds termwise for any finite linear combination
of merge operations.
\end{lemma}

\begin{proof}
For part~{\rm (i)}, suppose the label set decomposes as
\[
\{1,..,r\}=I\sqcup J,
\qquad
I,J\neq\varnothing,
\]
and the two groups are separated by a $2$-sphere.

Introduce commuting nilpotent variables
$
u_1,\ldots,u_r,
\quad
(u_i^2=0),
$
and form the moment generating polynomial
\[
\mathcal M_{\mathcal L}(u)
=
\sum_{T\subseteq[r]}
M(T)\prod_{i\in T}u_i,
\qquad
M(\varnothing)=1.
\]
The coefficient of
$
u_1\cdots u_r
$
in
$
\log\mathcal M_{\mathcal L}(u)
$
is precisely $\operatorname{Cum}(\mathcal L)$.

Since the two groups are split, every moment factors:
\[
M(T)
=
M(T\cap I)\,M(T\cap J).
\]
Hence we have
$
\mathcal M_{\mathcal L}(u)
=
\mathcal M_{\mathcal L_I}(u_I)
\mathcal M_{\mathcal L_J}(u_J),
$
and therefore
$
\log\mathcal M_{\mathcal L}
=
\log\mathcal M_{\mathcal L_I}
+
\log\mathcal M_{\mathcal L_J}.
$
Neither summand contains a monomial involving variables from both
$I$ and $J$.  Thus the coefficient of
$
u_1\cdots u_r
$
is zero which implies
$
\operatorname{Cum}(\mathcal L)=0.
$

For part~{\rm (ii)}, it is convenient to use the explicit
set-partition formula.  In
\[
\operatorname{Cum}(\mathcal L)
=
\sum_{\pi\in\Pi_r}
(-1)^{|\pi|-1}(|\pi|-1)!
\prod_{B\in\pi}M(B),
\]
a set partition $\pi$ contributes to the change only when $i$ and
$j$ belong to the same block.  Indeed, if they lie in different
blocks, each of those blocks omits one of the two labels, and by
hypothesis its moment does not change.

Suppose $i,j\in B_0\in\pi$.  Since there is only one block containing
both labels,
\[
\Delta
\left(
\prod_{B\in\pi}M(B)
\right)
=
c\,
M((B_0)_{i\sim j})
\prod_{\substack{B\in\pi\\B\neq B_0}}
M(B).
\]
Contracting $i$ and $j$ inside $B_0$ gives a bijection
\[
\left\{
\pi\in\Pi_r:
i,j\text{ lie in the same block}
\right\}
\longleftrightarrow
\Pi_{r-1},
\]
and this bijection preserves the number of blocks.  Hence
\[
\begin{aligned}
\Delta\operatorname{Cum}(\mathcal L)
&=
c
\sum_{\pi'\in\Pi_{r-1}}
(-1)^{|\pi'|-1}(|\pi'|-1)!
\prod_{B\in\pi'}M(B)
\\
&=
c\,\operatorname{Cum}(\mathcal L_{i\sim j}).
\end{aligned}
\]
Linearity gives the final assertion.
\end{proof}

\begin{lemma}\label{lem:ordinary-connected}
For a framed oriented link with $h\ge1$ ordinary labeled components,
\begin{equation}\label{eq:ordinary-connected}
 \operatorname{Cum}(\mathcal{L})\in z^{h-2}\Z[v,a^{\pm1}].
\end{equation}
\end{lemma}
\begin{proof}
Ordinary HOMFLY skein recursion gives $\mathcal H(\mathcal{L})\in z^{-h}\Z[v,a^{\pm1}]$.  For the connected estimate induct on $h$.  A knot has order at least $z^{-1}$.  For $h>1$, switch mixed crossings until one component splits off.  The endpoint cumulant vanishes.  At every crossing change the skein relation and Lemma \ref{lem:cumulant-contraction} give $z$ times an $(h-1)$-component cumulant, hence an element of $z\,z^{h-3}\Z[v,a^{\pm1}]$.
\end{proof}

We now replace some ordinary labels by normalized annular cycle
colors
$
B_m=\{m\}P_m.
$

Let $e$ denote the through-arc in Morton's relative annulus.
There are two natural actions of an annular skein element on $e$:
placing the annular element on the left or on the right.  We denote
these by
$
\ell(\,\cdot\,,e),
\quad
\rho(e,\,\cdot\,).
$

Morton's power-sum identity
\cite[Theorem~4.2]{MortonMurphy}
states
$
[P_m,e]
=
\{m\}w^m,
$
where $w^m$ is the relative skein element obtained by winding the
through-arc $m$ times through the annulus.
Multiplying by
$
B_m=\{m\}P_m
$
gives
\begin{equation}
\label{eq:relative-annulus-Bm}
\ell(B_m,e)-\rho(e,B_m)
=
\{m\}^2w^m.
\end{equation}

This identity is an equality in the relative annulus before any
exterior tangle is attached.  Hence it may be inserted into the
tubular neighborhood of an arbitrary framed companion.

Geometrically, after attachment to an exterior tangle, the term
$w^m$ merges the $B_m$-colored label into the ordinary through-arc.
Thus \eqref{eq:relative-annulus-Bm} is the colored analogue of the
ordinary skein contraction
\[
\mathcal L_+-\mathcal L_-=z\mathcal L_0:
\]
an ordinary mixed crossing contributes a factor $z$, whereas passing
an ordinary anchor through a normalized $B_m$ color contributes the
stronger factor
$
\{m\}^2.
$

For brevity, if a labeled framed link $\mathcal{L}$ has $k$ labels carrying the
normalized cycle colors
$B_{m_1},\ldots,B_{m_k}$,
and $h$ further ordinary labeled components, we write
$
\operatorname{Cum}(B_{m_1},\ldots,B_{m_k};\,h\ \mathrm{ordinary})
$
for its connected cumulant $\operatorname{Cum}(\mathcal{L})$.
The semicolon only separates the cycle-colored labels from the
ordinary labels; it does not define a new cumulant.

\begin{proposition}\label{prop:anchor}
For normalized colors $B_{m_1},\dots,B_{m_k}$ and $h\ge1$ ordinary labeled components,
\begin{equation}\label{eq:anchor}
 \operatorname{Cum}(B_{m_1},\dots,B_{m_k};h\text{ ordinary})
 \in\left(\prod_{i=1}^k\{m_i\}^2\right)z^{h-2}\Z[v,a^{\pm1}].
\end{equation}
\end{proposition}
\begin{proof}
Induct on $k+h$.  For $k=0$ use Lemma \ref{lem:ordinary-connected}.  If $k>0$, choose an ordinary component as an anchor and split it from the other labels by crossing changes.  At a passage through $B_{m_i}$, \eqref{eq:relative-annulus-Bm} supplies $\{m_i\}^2$ and merges that colored label into the ordinary anchor.  At an ordinary crossing the usual skein relation supplies $z$. Finally, the vanishing of the split
endpoint proves
\eqref{eq:anchor}.
\end{proof}

\subsection{The sharp gcd denominator}
\begin{theorem}\label{thm:gcd}
Let $\mathcal L$ have an arbitrary integer framing vector $\vec\tau$, and let $\vec\lambda$ be a nonempty cycle tuple with $r$ total parts and part-gcd $D=D_{\vec\lambda}$.  For the signed $q$-framed partition function,
\begin{equation}\label{eq:gcd}
 \check F^{\vec\tau}_{\vec\lambda}\in
 \frac{\{\vec\lambda\}^2}{\{D\}^2}\Z[v,a^{\pm1}],
 \qquad
 H^{\vec\tau}_{\vec\lambda}\in\Z[v,a^{\pm1}].
\end{equation}
\end{theorem}
\begin{proof}
List the $r$ cycle parts of $\vec\lambda$ as $n_1,\ldots,n_r$, retaining their original link-component assignments, and put $D=\gcd(n_1,\ldots,n_r)$.
Using Lemma~\ref{lem:framed-parallel}.  The framing correction contributed by all labeled cycles is the common Laurent monomial
\[
 u_{\vec\lambda,\vec\tau}
 =(-a^{-1})^{\sum_{j=1}^r\tau_{c(j)}n_j},
\]
where $c(j)$ is the original component carrying the $j$th cycle.  Since every labeled cycle occurs exactly once in every set-partition term, this factor is independent of the block decomposition and is therefore irrelevant to divisibility.

Fix $i$.  Expand only the normalized color $B_{n_i}$ by the integral Newton representative
\[
 b_{n_i}=z\sum_{\nu=0}^{n_i-1}\beta_{i,\nu},
\]
where each $\beta_{i,\nu}$ is an integral Hecke word whose annular closure is one ordinary framed component.  By linearity of the cumulant and Proposition~\ref{prop:anchor}, applied with one ordinary label and the remaining $r-1$ normalized cycle labels,
\[
 \operatorname{Cum}\!\left(\operatorname{cl}(\beta_{i,\nu});B_{n_j},\ j\ne i\right)
 \in\left(\prod_{j\ne i}\{n_j\}^2\right)z^{-1}\Z[v,a^{\pm1}].
\]
The explicit factor $z$ in $b_{n_i}$ cancels the knot pole $z^{-1}$.  Hence, for every $i$,
\begin{equation}\label{eq:gcd-step-one}
 \check F^{\vec\tau}_{\vec\lambda}
 \in\left(\prod_{j\ne i}\{n_j\}^2\right)\Z[v,a^{\pm1}].
\end{equation}

Now, set $A_j(v)=\{n_j\}^2\in\Z[v]$ and $T(v)=\prod_jA_j(v)=\{\vec\lambda\}^2$.  The divisors in \eqref{eq:gcd-step-one} are $G_i=T/A_i$.  Since every $G_i$ is monic, divisibility by all $G_i$ implies divisibility by their monic least common multiple, denoted by $\operatorname{lcm}_i(G_i)$.  For an irreducible polynomial $\phi(v)\in\Q[v]$, write $a_j=\operatorname{ord}_\phi A_j$.  Then
\[
 \operatorname{ord}_\phi\operatorname{lcm}_i(G_i)
 =\max_i\left(\sum_j a_j-a_i\right)
 =\sum_j a_j-\min_i a_i.
\]
Thus, we obtain
\[
 \operatorname{lcm}_i(G_i)
 =\frac{\prod_jA_j}{\gcd(A_1,\ldots,A_r)}.
\]
Now $\{m\}=q^{-m}(q^{2m}-1)$ and
$\gcd(q^{2m}-1,q^{2n}-1)=q^{2\gcd(m,n)}-1$.  Up to Laurent monomial units this gives
$\gcd(\{m\},\{n\})=\{\gcd(m,n)\}$.  Since the squares $\{m\}^2$ are monic polynomials in $v$, it follows that
\[
 \gcd\bigl(\{n_1\}^2,\ldots,\{n_r\}^2\bigr)=\{D\}^2.
\]
Consequently, we have
\[
 \operatorname{lcm}_{1\le i\le r}\prod_{j\ne i}\{n_j\}^2
 =\frac{\{\vec\lambda\}^2}{\{D\}^2}.
\]
Therefore, it proves the first assertion in \eqref{eq:gcd}.

By definition,
\(
 H^{\vec\tau}_{\vec\lambda}
 =\frac{\{D\}^2\check F^{\vec\tau}_{\vec\lambda}}{\{\vec\lambda\}^2},
\)
so the preceding divisibility immediately gives
$H^{\vec\tau}_{\vec\lambda}\in\Z[v,a^{\pm1}]$.
\end{proof}

\begin{remark}\label{rem:q1-not-enough}
Theorem \ref{thm:gcd} is only a denominator estimate.  The cancellation of the remaining factor $U_D(v)$ will be obtained later from Theorem~\ref{thm:framed-congruence}; see Section~\ref{sec:rational}.
\end{remark}

\section{Local Hecke identities}\label{sec:hecke}

This section isolates the two local Hecke-algebra identities used in
Section~\ref{sec:exterior}. Since their proofs are lengthy, we collect
the representation-theoretic and relative-cocenter calculations in
Appendix~\ref{app:local-hecke}; here we record only the necessary notation
and the two final formulas.

\subsection{Setup and notation}
Fix $n\ge2$ and a root $\alpha$ with $\ord(\alpha^2)=n$.  Let
\[
 R_\alpha:=\mathcal O_\alpha
   =\Q(\alpha)[q^{\pm1}]_{(q-\alpha)}.
\]
Then $\{j\}$ is a unit of $R_\alpha$ for $0<j<n$, whereas $\{n\}$ has a simple zero at
$q=\alpha$; equivalently,
\[
 \{n\}R_\alpha=(q-\alpha)R_\alpha.
\]
Let
\[
 \HH:=\HH_n(R_\alpha),\qquad
 \GG:=\HH_{2n}(R_\alpha),\qquad
 A:=\HH_L\otimes_{R_\alpha}\HH_R\subset\GG.
\]
Here $\HH_L$ and $\HH_R$ are the Hecke algebras of the left and right $n$-strand packets.
Let $b=b_n\in\HH$ be the integral Newton element of \eqref{eq:newton-coupon}, write
$b_L=b\otimes1$ and $b_R=1\otimes b$, and let $C=T_d\in\GG$ be the positive
standard-basis lift of the permutation $d$ exchanging the two packets.  Finally put
\[
 \eps_n=(-1)^{n-1},\qquad
 [A,\GG]=\operatorname{Span}_{R_\alpha}\{aY-Ya:\ a\in A,\ Y\in\GG\}.
\]
The module $[A,\GG]$ is a linear commutator submodule, not an ideal.  This distinction is
important in Appendix~\ref{app:local-hecke}: every multiplication of a relative-cocenter
class there is justified by an explicit commutation, transport, or transfer relation.

\subsection{The two crossing identities and their role}
Define the self- and mixed-crossing remainders
\[
 S_n:=b_L(C-C^{-1})-\eps_n b_Lb_R,
 \qquad
 M_n:=b_Lb_R(C^2-1)-\eps_n n\{n\}^2b_LC.
\]

\begin{theorem}[Local crossing identities]\label{thm:relative-hecke}
For every $n\ge2$ and every $\alpha$ with $\ord(\alpha^2)=n$,
\begin{align}
 S_n&\in[A,\GG]+\{n\}^2\GG,\label{eq:self-main}\\
 M_n&\in[A,\GG]+\{n\}^4\GG.\label{eq:mixed-main}
\end{align}
\end{theorem}

The proof is given in Appendix~\ref{app:local-hecke}.
The two local crossing identities \eqref{eq:self-main} and
\eqref{eq:mixed-main} generalize, at the level of local Hecke
algebras, the two congruence skein relations proposed in our
previous work~\cite[Eqs.~(3) and~(4)]{CLPZ}.
In that work, we showed that these conjectural relations imply
the Hecke lifting formula; see~\cite[Eq.~(5) and
Proposition~1.7]{CLPZ}.
Here, the two local crossing identities are used to prove
Theorem~\ref{thm:equal-cycle}, which may be viewed as a
generalization of that Hecke lifting formula, formulated locally at roots of unity and valid for every integer $n\ge2$.

\section{From local identities to equal-cycle congruences}\label{sec:exterior}

The identities of Section~\ref{sec:hecke} become link congruences only after two points
are checked: the exterior pairing must annihilate the relative commutators, and the
integer framings created by oriented smoothing must be tracked consistently.  Throughout
this section $\check Z_m(K;\tau)$ denotes the \emph{ordinary framed} normalized cycle
evaluation; the signed moments of Section~\ref{sec:prelim} continue to carry the
superscript $\vec\tau$.

\subsection{Transport along exterior ribbons}
\begin{lemma}\label{lem:exterior-transport}
Remove a parallel-oriented crossing box from a framed link and replace each exterior arc
by $n$ parallel strands.  If the exterior is packet-preserving, it commutes with the
packet algebra $A$.  If it exchanges the packets, it intertwines the two $\HH_n$-actions.
These identities remain valid after adding normalized spectator colors supported away
from the two packet ribbons.
\end{lemma}
\begin{proof}
Move an internal Hecke coupon along the product tubular neighborhood of the corresponding
exterior ribbon.  The ribbon framing identifies the transverse disks, so this motion is a
framed isotopy.  A packet-preserving exterior returns the coupon to the same packet; a
packet-exchanging exterior returns it to the opposite packet.
\end{proof}

If $S$ is packet-preserving, define the closure pairing
\[
 \Phi_S:\GG\longrightarrow\mathcal O_\alpha[a^{\pm1}],\qquad
 \Phi_S(Y)=\Tr(YS).
\]
Lemma~\ref{lem:exterior-transport} gives
\begin{equation}\label{eq:exterior-commutator}
 \Tr([a,Y]S)=0,\qquad a\in A,\quad Y\in\GG.
\end{equation}
Thus the pairing factors through $\GG/[A,\GG]$.  If the exterior exchanges the packets, write
$E=CS$ with $S$ packet-preserving.  Then
\begin{equation}\label{eq:mixed-exterior-conversion}
 (C-C^{-1})E=(C^2-1)S,
\end{equation}
which is the reason the mixed identity is formulated using $C^2-1$.

\begin{lemma}\label{lem:regular-closure}
Assume $\ord(\alpha^2)=n>1$.  Closure of a finite ordinary framed local tangle, possibly
with spectator colors $B_m$, lies in $\mathcal O_\alpha[a^{\pm1}]$.  More generally, an
outside factor $(q-\alpha)^s$ cannot be lowered by closure.
\end{lemma}
\begin{proof}
Expand every spectator $B_m$ by the integral Newton representative.  Ordinary HOMFLY--PT
skein reduction then introduces only powers of $z^{-1}$.  Since
$z(\alpha)=\alpha-\alpha^{-1}\ne0$ for $n>1$, these denominators are units in
$\mathcal O_\alpha$.
\end{proof}

\subsection{Framing factors}
Let $K$ be a knot with integer framing $\tau$.  Define the ordinary framed normalized
$m$-cycle evaluation by
\begin{equation}\label{eq:ordinary-framed-cycle}
 \check Z_m(K;\tau)
 :=\{m\}\mathcal H(K^\tau\star P_m;q,a)
 =a^{m\tau}\{m\}\sum_{A\vdash m}\chi_A((m))q^{\tau\kappa_A}W_A(K).
\end{equation}
For a single cycle, Lemma~\ref{lem:framed-parallel} gives the precise relation with the
signed moment of Section~\ref{sec:prelim}:
\begin{equation}\label{eq:ordinary-signed-cycle}
 \check Z^{\tau}_{(m)}(K)=(-a^{-1})^{m\tau}\check Z_m(K;\tau).
\end{equation}

\begin{lemma}[Framing congruence for an $n$-cycle]\label{lem:root-framing}
If $\ord(\alpha^2)=n$ and $\eps_n=(-1)^{n-1}$, then for every integer $\tau$,
\begin{equation}\label{eq:root-framing}
 \check Z_n(K;\tau)
 \equiv a^{n\tau}\eps_n^{\tau}\check Z_n(K;0)
 \pmod{(q-\alpha)^2}.
\end{equation}
\end{lemma}
\begin{proof}
This lemma can be regarded as a generalization of the prime-cycle framing congruence in \cite{CLPZ}, but the proof is the same.  The case $\tau=0$ is immediate.  For $\tau\ne0$, character inversion and the
standard hook-character identity give
\begin{equation}\label{eq:hook-frame}
 \check Z_n(K;\tau)
 =a^{n\tau}\frac{\{n\}}{\{n\tau\}}
 \sum_{\mu\vdash n}\frac{\check Z^0_\mu(K)}{\mathfrak z_\mu}
 \prod_{i=1}^{\ell(\mu)}\frac{\{n\tau\mu_i\}}{\{\mu_i\}}.
\end{equation}
If $\mu\ne(n)$, then $\ell(\mu)\ge2$ and every part satisfies $0<\mu_i<n$.
Hence $\{\mu_i\}(\alpha)\ne0$, while $\{n\tau\mu_i\}(\alpha)=0$, so the
corresponding term has order at least two at $q=\alpha$.

For $\mu=(n)$ the coefficient is
\[
 R_{n,\tau}(q)=\frac1n\frac{\{n^2\tau\}}{\{n\tau\}}.
\]
Putting $x=q^{n\tau}$ gives
\[
 R_{n,\tau}(q)=\frac1n\frac{x^n-x^{-n}}{x-x^{-1}}.
\]
At $q=\alpha$, $x=\alpha^{n\tau}=\pm1$, and
$R_{n,\tau}(\alpha)=\eps_n^\tau$.  The Laurent polynomial
$(x^n-x^{-n})/(x-x^{-1})$ is invariant under $x\mapsto x^{-1}$; hence its first
derivative vanishes at the fixed points $x=\pm1$.  Therefore
$R_{n,\tau}(q)=\eps_n^\tau+O((q-\alpha)^2)$, proving the lemma.
\end{proof}

\begin{lemma}[Framing under oriented smoothing]\label{lem:smoothing-framing}
Let $c$ be a crossing of sign $\sigma(c)\in\{\pm1\}$ in an oriented framed diagram,
and give the oriented smoothing the inherited ribbon framing.
\begin{enumerate}[label=(\roman*)]
\item If $c$ is a mixed crossing between distinct components $K_i,K_j$ with framings
$\tau_i,\tau_j$, the smoothing merges them into $K_{ij}$ and
\begin{equation}\label{eq:mixed-smoothing-framing}
 \tau(K_{ij})=\tau_i+\tau_j+2\operatorname{lk}(K_i,K_j)-\sigma(c),
\end{equation}
so
\begin{equation}\label{eq:mixed-smoothing-parity}
 \tau(K_{ij})\equiv\tau_i+\tau_j+1\pmod2.
\end{equation}
\item If $c$ is a self-crossing of $K$ and the smoothing splits it into $K',K''$, then
\begin{equation}\label{eq:self-smoothing-framing}
 \tau(K)=\tau(K')+\tau(K'')+2\operatorname{lk}(K',K'')+\sigma(c),
\end{equation}
and hence
\begin{equation}\label{eq:self-smoothing-parity}
 \tau(K)\equiv\tau(K')+\tau(K'')+1\pmod2.
\end{equation}
\end{enumerate}
Thus every oriented smoothing changes the parity of the sum of the individual framings by
exactly one.
\end{lemma}
\begin{proof}
For blackboard framings the formulas follow by comparing self-writhes.  In the mixed case,
after smoothing, every remaining crossing between $K_i$ and $K_j$ becomes a self-crossing
of $K_{ij}$; their signed sum is
$2\operatorname{lk}(K_i,K_j)-\sigma(c)$.  In the self-crossing case, after smoothing,
the crossings between $K'$ and $K''$ have signed sum $2\operatorname{lk}(K',K'')$,
and restoring the removed crossing contributes $\sigma(c)$.  For arbitrary ribbon
framings, the additional full twists add under merging and split additively under
separation.  Reduction modulo two gives the parity statements.
\end{proof}

\subsection{The knot case}
More generally, for a framed diagram $D$ of an $s$-component labeled link with one
$m_i$-cycle on the $i$th component, write
\[
 \check Z_{(m_1,\ldots,m_s)}(D)
 :=\left(\prod_{i=1}^s\{m_i\}\right)
\mathcal H\bigl(D\star(P_{m_1},\ldots,P_{m_s})\bigr).
\]
This is always the ordinary framed skein realization.

\begin{proposition}[Knot congruence]\label{prop:uniform-knot}
If $\ord(\alpha^2)=n$, then every framed knot $(K,\tau)$ satisfies
\begin{equation}\label{eq:uniform-knot}
 \check Z_n(K;\tau)
 \equiv\eps_n^\tau\Psi_n\check Z_1(K;\tau)
 \pmod{(q-\alpha)^2}.
\end{equation}
\end{proposition}
\begin{proof}
Let $D_+,D_-,D_0$ be a self-crossing triple, with $D_0=K'\cup K''$ carrying the
inherited framings.  The self-crossing identity \eqref{eq:self-main}, together with
Lemmas~\ref{lem:exterior-transport} and \ref{lem:regular-closure}, gives
\begin{equation}\label{eq:knot-upper-recurrence}
 \check Z_n(D_+)-\check Z_n(D_-)
 \equiv\eps_n\check Z_{(n,n)}(D_0)
 \pmod{(q-\alpha)^2}.
\end{equation}
By the moment--cumulant formula and Theorem~\ref{thm:gcd},
\begin{equation}\label{eq:knot-upper-factorization}
 \check Z_{(n,n)}(D_0)
 \equiv\check Z_n(K';\tau')\check Z_n(K'';\tau'')
 \pmod{(q-\alpha)^2}.
\end{equation}
For the fundamental cycle, the HOMFLY--PT skein relation gives the exact identity
\[
 \check Z_1(D_+)-\check Z_1(D_-)=\check Z_{(1,1)}(D_0).
\]
Again Theorem~\ref{thm:gcd} gives
$\check Z_{(1,1)}(D_0)-\check Z_1(K';\tau')\check Z_1(K'';\tau'')\in z^2\Z[v,a^{\pm1}]$.
Applying $\Psi_n$ and using $\Psi_n(z^2)=\{n\}^2$ yields
\begin{equation}\label{eq:knot-lower-factorization}
 \Psi_n\check Z_{(1,1)}(D_0)
 \equiv
 \Psi_n\check Z_1(K';\tau')\,\Psi_n\check Z_1(K'';\tau'')
 \pmod{(q-\alpha)^2}.
\end{equation}
If the induction hypothesis is applied to $K'$ and $K''$, the leading term in
\eqref{eq:knot-upper-recurrence} acquires the sign
$\eps_n^{1+\tau'+\tau''}$.  By
\eqref{eq:self-smoothing-parity}, this equals $\eps_n^{\tau(D)}$.  Thus both sides of
\eqref{eq:uniform-knot} satisfy the same crossing-change recursion.  Switching crossings
to a descending diagram reduces to a framed unknot, where the claim follows from
Lemma~\ref{lem:root-framing} and
$\check Z_n(U;0)=a^n-a^{-n}=\Psi_n\check Z_1(U;0)$.
\end{proof}

\subsection{Equal cycle lengths}
For an ordinary framed labeled link $\mathcal L=K_1\cup\cdots\cup K_r$, let
$H^{\mathrm{ord}}_{(m_1,\ldots,m_r)}(\mathcal L)$ denote the connected normalization of
\eqref{eq:Hlambda-def} built from the ordinary framed moments above.  In particular,
\[
 H^{\mathrm{ord}}_{(n^r)}(\mathcal L)=\frac{\check F^{\mathrm{ord}}_{(n^r)}(\mathcal L)}{\{n\}^{2r-2}},
 \qquad
 H^{\mathrm{ord}}_{(1^r)}(\mathcal L)=\frac{\check F^{\mathrm{ord}}_{(1^r)}(\mathcal L)}{z^{2r-2}}.
\]

\begin{theorem}[Equal-cycle congruence]\label{thm:equal-cycle}
Let $\mathcal L$ be a framed link with $r\ge1$ labeled components, each carrying $B_n$, and put
$\tau(\mathcal L)=\sum_i\tau(K_i)$.  If $\ord(\alpha^2)=n$, then
\begin{equation}\label{eq:equal-cycle}
 H^{\mathrm{ord}}_{(n^r)}(\mathcal L)
 \equiv n^{r-1}\eps_n^{\tau(\mathcal L)}
 \Psi_nH^{\mathrm{ord}}_{(1^r)}(\mathcal L)
 \pmod{(q-\alpha)^2}.
\end{equation}
\end{theorem}
\begin{proof}
For $r=1$ this is Proposition~\ref{prop:uniform-knot}.  Suppose $r>1$ and choose a mixed
crossing.  The mixed identity \eqref{eq:mixed-main} has leading term
$\eps_n n\{n\}^2$ times the contracted label and remainder divisible by $\{n\}^4$.
Inside the $r$-label cumulant, the remaining $r-2$ normalized cycle labels contribute
$\{n\}^{2(r-2)}$ by Proposition~\ref{prop:anchor}; hence the total unnormalized error is
divisible by $\{n\}^{2r}$.  Dividing by the connected normalization
$\{n\}^{2r-2}$ leaves a second-order error and gives
\begin{equation}\label{eq:equal-cycle-recurrence}
 H^{\mathrm{ord}}_{(n^r)}(\mathcal L_+)-H^{\mathrm{ord}}_{(n^r)}(\mathcal L_-)
 \equiv\eps_n n\,H^{\mathrm{ord}}_{(n^{r-1})}(\mathcal L_0)
 \pmod{(q-\alpha)^2}.
\end{equation}
For the fundamental cycle, contraction of connected cumulants gives exactly
\[
 H^{\mathrm{ord}}_{(1^r)}(\mathcal L_+)-H^{\mathrm{ord}}_{(1^r)}(\mathcal L_-)
 =H^{\mathrm{ord}}_{(1^{r-1})}(\mathcal L_0).
\]
Apply the induction hypothesis to $\mathcal L_0$.  By
\eqref{eq:mixed-smoothing-parity},
$\tau(\mathcal L_0)\equiv\tau(\mathcal L)+1\pmod2$, so the extra local factor $\eps_n$ in
\eqref{eq:equal-cycle-recurrence} changes $\eps_n^{\tau(\mathcal L_0)}$ into
$\eps_n^{\tau(\mathcal L)}$.  Hence the difference between the two sides of
\eqref{eq:equal-cycle} is invariant modulo $(q-\alpha)^2$ under mixed crossing changes.
Switch mixed crossings until the labeled link splits into two nonempty groups; both
connected cumulants then vanish by Lemma~\ref{lem:cumulant-contraction}.  The result follows.
\end{proof}

\section{Cabling and arbitrary cycle lengths}\label{sec:cable}

We now pass from equal cycle lengths to arbitrary cycle tuples and complete the proof of
Theorem~\ref{thm:framed-congruence}.  The algebraic extension below is introduced only
to interpolate the cable slope; no fractional powers of $q$ or $a$ occur.

\subsection{Cycle form of the cable formula}
For integers $m,D\ge1$ and $k\in\Z$, define
\begin{equation}\label{eq:algebraic-framed-cycle}
 \check Z_{mD}\!\left(K;\frac{k}{D}\right)
 :=a^{mk}\{mD\}\sum_{A\vdash mD}\chi_A((mD))
 q^{k\kappa_A/D}W_A(K).
\end{equation}
The cycle character is supported on hooks
$A=(mD-b,1^b)$, for which
$\kappa_A=mD(mD-1-2b)$; hence $k\kappa_A/D\in\Z$.
If $k/D$ is an integer, \eqref{eq:algebraic-framed-cycle} is exactly the ordinary framed
skein evaluation of \eqref{eq:ordinary-framed-cycle}.

Assume now $(k,D)=1$ and let $J=C_{D,k}(K)$ be the connected $(D,k)$ cable of a
zero-framed knot $K$, taken in the zero longitude and reset to its own zero framing.  Write
\[
 s_A(x^D)=\sum_Bc^B_{A;D}s_B(x).
\]
In our framing-independent normalization the Rosso--Jones cable formula is given in
\cite[Theorem~13]{MM}; see also the original quantum-group formula \cite{RJ}.
\begin{equation}\label{eq:cable-Schur}
 W_A(J)=q^{-kD\kappa_A}a^{-k(D-1)|A|}
 \sum_Bc^B_{A;D}q^{k\kappa_B/D}W_B(K).
\end{equation}

\begin{lemma}\label{lem:cable-cycle}
For every $m\ge1$,
\begin{equation}\label{eq:cable-cycle}
 \check Z_{mD}\!\left(K;\frac{k}{D}\right)
 =\frac{\{mD\}}{\{m\}}\check Z_m(J;kD).
\end{equation}
\end{lemma}
\begin{proof}
Frobenius' formula and the plethystic identity $p_m(x^D)=p_{mD}(x)$ give
\[
 \sum_{A\vdash m}\chi_A((m))c^B_{A;D}=\chi_B((mD)).
\]
Substituting \eqref{eq:cable-Schur} into the definition of $\check Z_m(J;kD)$
cancels the factors $q^{kD\kappa_A}$, while the $a$-factor simplifies as
$a^{mkD}a^{-mk(D-1)}=a^{mk}$.  The preceding character identity then yields
\eqref{eq:cable-cycle}.
\end{proof}

\subsection{Connected cable factors}
Let $\mathcal{L}=K_1\cup\cdots\cup K_r$ be zero-individually-framed, let $D_i\ge1$, and put
$D_*=\gcd(D_1,\ldots,D_r)$.  For a fixed $m$ define the connected normalization
\begin{equation}\label{eq:connected-algebraic-normalization}
 H^{\mathrm{ord}}_{(mD_i)}\!\left(\mathcal{L};\frac{k_i}{D_i}\right)
 :=\frac{\{mD_*\}^2}{\prod_{i=1}^r\{mD_i\}^2}
 \check F^{\mathrm{ord}}_{(mD_i)}\!\left(\mathcal{L};\frac{k_i}{D_i}\right).
\end{equation}
If $(k_i,D_i)=1$ for every $i$, cable the $i$th component by $C_{D_i,k_i}$ and let
$J$ be the resulting $r$-component link.  Since \eqref{eq:cable-cycle} is compatible with
deleting labels, it factors out of every block of the moment--cumulant formula and yields
\begin{equation}\label{eq:connected-cable-factor}
 H^{\mathrm{ord}}_{(mD_i)}\!\left(\mathcal{L};\frac{k_i}{D_i}\right)
 =\mathcal P_m(D_1,\ldots,D_r)\,
 H^{\mathrm{ord}}_{(m^r)}(J;k_iD_i),
\end{equation}
where
\begin{equation}\label{eq:cable-prefactor}
 \mathcal P_m(D_1,\ldots,D_r)
 =\frac{\{mD_*\}^2\{m\}^{r-2}}{\prod_{i=1}^r\{mD_i\}}.
\end{equation}
In particular,
\begin{equation}\label{eq:cable-prefactor-Adams}
 \mathcal P_n(D_1,\ldots,D_r)=\Psi_n\mathcal P_1(D_1,\ldots,D_r).
\end{equation}
Moreover $\mathcal P_n$ is a unit at a root with $\ord(\alpha^2)=n$, since each ratio
$\{nD\}/\{n\}$ tends to $D(\alpha^n)^{D-1}\ne0$.

\subsection{Finite dependence on the cable slopes}
\begin{lemma}\label{lem:finite-slope}
Fix $m,D_1,\ldots,D_r$ and the zero-individually-framed labeled link $\mathcal{L}$.  There are a
finite index set $\mathcal I$ and coefficients $R_I(q,a)$, independent of the slopes
$k_1,\ldots,k_r$, such that
\begin{equation}\label{eq:finite-slope}
 H^{\mathrm{ord}}_{(mD_i)}\!\left(\mathcal{L};\frac{k_i}{D_i}\right)
 =a^{m\sum_i k_i}\sum_{I\in\mathcal I}R_I(q,a)
 q^{\sum_i k_i\kappa_{A_i(I)}/D_i}.
\end{equation}
Every exponent $\kappa_{A_i(I)}/D_i$ occurring in a nonzero term is an integer.
\end{lemma}
\begin{proof}
Expand each cycle moment by the finite character sum and the connected coefficient by the
finite set-partition formula.  Every label contributes the common factor $a^{mk_i}$, and
all remaining slope dependence is through $q^{k_i\kappa_{A_i}/D_i}$.  A nonzero
$(mD_i)$-cycle character forces $A_i=(mD_i-b_i,1^{b_i})$, so
\[
 \frac{\kappa_{A_i}}{D_i}=m(mD_i-1-2b_i)\in\Z.
\]
\end{proof}

\begin{lemma}\label{lem:hook-root-phase}
Let $\ord(\alpha^2)=n$ and $\omega=\alpha^n\in\{\pm1\}$.  For a hook
$A=(nD-b,1^b)$ and a hook $B=(D-c,1^c)$,
\begin{align}
 \alpha^{k\kappa_A/D}&=\omega^{k(nD-1)},\label{eq:upper-hook-phase}\\
 \alpha^{nk\kappa_B/D}&=\omega^{k(D-1)},\label{eq:lower-hook-phase}
\end{align}
independently of $b,c$, and
\begin{equation}\label{eq:phase-ratio}
 \omega^{k(nD-1)}=\eps_n^{kD}\omega^{k(D-1)}.
\end{equation}
\end{lemma}
\begin{proof}
Use $\kappa_A=nD(nD-1-2b)$ and $\kappa_B=D(D-1-2c)$.  Since $\omega^2=1$, the hook
labels disappear.  Finally $\omega^{n-1}=\eps_n$: this is immediate if $n$ is odd, while
for even $n$ exact squared order gives $\omega=-1$.
\end{proof}

\begin{lemma}\label{lem:interpolation}
Let $F$ be a field of characteristic zero and suppose
\[
 f(s;k)=\sum_{\nu=1}^MR_\nu(s)(1+s)^{\Lambda_\nu(k)},\qquad k\in\Z^r,
\]
where the Laurent series $R_\nu(s)\in F((s))$ have a common finite lower bound and the
$\Lambda_\nu$ are integer-valued affine linear forms.  If $\operatorname{ord}_s f(s;k)\ge b$
on a Cartesian product of infinite subsets of $\Z$, then the same inequality holds for all
$k\in\Z^r$.
\end{lemma}
\begin{proof}
For every degree $d<b$, the coefficient $[s^d]f(s;k)$ is a polynomial in the $k_i$ because
$[s^m](1+s)^{\Lambda(k)}=\binom{\Lambda(k)}m$.  A polynomial over a characteristic-zero
field that vanishes on a product of infinite subsets is identically zero.
\end{proof}

\subsection{Arbitrary cycle lengths}
\begin{theorem}\label{thm:universal-slope}
Let $D_1,\ldots,D_r$ be positive integers.  Fix $n\ge2$ and
$\ord(\alpha^2)=n$.  For every integer vector $(k_1,\ldots,k_r)$,
\begin{equation}\label{eq:universal-slope}
 H^{\mathrm{ord}}_{(nD_i)}\!\left(\mathcal{L};\frac{k_i}{D_i}\right)
 \equiv n^{r-1}\eps_n^{\sum_i k_iD_i}
 \Psi_nH^{\mathrm{ord}}_{(D_i)}\!\left(\mathcal{L};\frac{k_i}{D_i}\right)
 \pmod{(q-\alpha)^2}.
\end{equation}
\end{theorem}
\begin{proof}
First assume $k_i\ne0$ and $(k_i,D_i)=1$.  Apply
\eqref{eq:connected-cable-factor} with $m=n$ and $m=1$.  By
\eqref{eq:cable-prefactor-Adams}, both sides have the same cable prefactor after applying
$\Psi_n$.  The cabled components have integer framings $k_iD_i$, so
Theorem~\ref{thm:equal-cycle} gives \eqref{eq:universal-slope} on the product of the
coprime slope grids.

Put $q=\alpha(1+s)$.  By Lemma~\ref{lem:finite-slope}, after factoring out the common
$a$-monomial, every upper summand has root phase
$\rho(k)=\omega^{\sum_i k_i(nD_i-1)}$.  Lemma~\ref{lem:hook-root-phase} shows that the
lower root phase, after multiplication by $\eps_n^{\sum_i k_iD_i}$, is the same
$\rho(k)$.  Dividing by these nonzero common factors leaves a finite sum of the form in
Lemma~\ref{lem:interpolation}.  Since the coprime set
$\{k\in\Z\setminus\{0\}:(k,D_i)=1\}$ is infinite for every $i$, interpolation extends the
order-two identity to all integer slopes.
\end{proof}

\begin{proof}[Proof of Theorem~\ref{thm:framed-congruence}]
List all cycle parts of $\vec\nu$ as $D_1,\ldots,D_r$, retaining the original link-component
assignment, and let $\tau_i$ be the framing of the component carrying $D_i$.  In the ordinary framed parallel realization of Lemma~\ref{lem:framed-parallel}, reset the underlying parallel companions to zero individual framing and encode each original integer framing through the algebraic slope parameter.  Thus Theorem~\ref{thm:universal-slope} applies with
\[
 k_i=D_i\tau_i.
\]
Put $S=\sum_i\tau_iD_i$.  Since $D_i^2\equiv D_i\pmod2$, the framing phase in
\eqref{eq:universal-slope} is $\eps_n^S$, and hence
\begin{equation}\label{eq:ordinary-final-comparison}
 H^{\mathrm{ord}}_{n\vec\nu}
 \equiv n^{r-1}\eps_n^S\Psi_nH^{\mathrm{ord}}_{\vec\nu}
 \pmod{(q-\alpha)^2}.
\end{equation}
Lemma~\ref{lem:framed-parallel} gives
\begin{equation}\label{eq:signed-ordinary-relation}
 H^{\vec\tau}_{\vec\nu}=(-a^{-1})^S H^{\mathrm{ord}}_{\vec\nu},
 \qquad
 H^{\vec\tau}_{n\vec\nu}=(-a^{-1})^{nS}H^{\mathrm{ord}}_{n\vec\nu}.
\end{equation}
Finally,
\[
 (-a^{-1})^{nS}\eps_n^S
 =\Psi_n\bigl((-a^{-1})^S\bigr).
\]
Multiplying \eqref{eq:ordinary-final-comparison} by $(-a^{-1})^{nS}$ and using
\eqref{eq:signed-ordinary-relation} proves
\[
 H^{\vec\tau}_{n\vec\nu}
 \equiv n^{r-1}\Psi_nH^{\vec\tau}_{\vec\nu}
 \pmod{(q-\alpha)^2},
\]
which is Theorem~\ref{thm:framed-congruence}.
\end{proof}

\section{Rational polynomiality}\label{sec:rational}

We now combine Theorem~\ref{thm:gcd} with Theorem~\ref{thm:framed-congruence} to prove the following.
\begin{proposition}[Rational polynomiality]\label{prop:rational-polynomiality}
For every nonempty tuple of partitions $\vec A$, we have
\[
 vP^{\vec\tau}_{\vec A}\in\mathbb Q[v,a^{\pm1}].
\]
\end{proposition}
\begin{proof}
Fix a nonempty partition tuple $\vec\lambda$ with $D=D_{\vec\lambda}$ and $r=\ell(\vec\lambda)$.  By Theorem~\ref{thm:gcd}, every $H^{\vec\tau}_{\vec\rho}$ belongs to $\mathbb Z[v,a^{\pm1}]$.  By \eqref{eq:intro-key-normalization}, it is enough to prove
\begin{equation}\label{eq:mobius-div}
 U_D(v)\ \Bigm|\
 \sum_{k\mid D}\mu(k)k^{r-1}\Psi_kH^{\vec\tau}_{\vec\lambda/k}
 \qquad\text{in }\mathbb Q[v,a^{\pm1}].
\end{equation}

Fix $e\mid D$ ($e>1$) with $\operatorname{rad}(e)=\prod_{p\mid e}p$, and let $\alpha$ satisfy $\ord(\alpha^2)=e$.  Only squarefree divisors contribute to the sum in \eqref{eq:mobius-div} since $\mu(k)=0$ unless $k$ is squarefree.  Write such a divisor uniquely as
\[
 k=d\ell,\qquad d\mid\operatorname{rad}(e),\qquad (\ell,e)=1.
\]

By the following Lemma \ref{lem:Mobius-term-reduction}, which is an easy consequence of Theorem~\ref{thm:framed-congruence}, we obtain 

\[
 \Psi_{d\ell}H^{\vec\tau}_{\vec\lambda/(d\ell)}
 \equiv (e/d)^{r-1}\Psi_{e\ell}H^{\vec\tau}_{\vec\lambda/(e\ell)}
 \pmod{(q-\alpha)^2}.
\]

 Multiplication by $\mu(d \ell)(d\ell)^{r-1}$ gives
\[
\begin{aligned}
&
\mu(d\ell)
(d\ell)^{r-1}
\Psi_{d\ell}
H_{\vec\lambda/(d\ell)}^{\vec\tau}
\\
&\qquad\equiv
\mu(d)\mu(\ell)
(d\ell)^{r-1}
\left(\frac{e}{d}\right)^{r-1}
\Psi_{e\ell}
H_{\vec\lambda/(e\ell)}^{\vec\tau}
\\
&\qquad=
\mu(d)\,
\mu(\ell)
(e\ell)^{r-1}
\Psi_{e\ell}
H_{\vec\lambda/(e\ell)}^{\vec\tau}
\pmod{(q-\alpha)^2}.
\end{aligned}
\]
Therefore, we have 
\[
\begin{aligned}
&
\sum_{d\mid\operatorname{rad}(e)}
\mu(d\ell)
(d\ell)^{r-1}
\Psi_{d\ell}
H_{\vec\lambda/(d\ell)}^{\vec\tau}
\\
&\qquad\equiv
\mu(\ell)
(e\ell)^{r-1}
\Psi_{e\ell}
H_{\vec\lambda/(e\ell)}^{\vec\tau}
\sum_{d\mid\operatorname{rad}(e)}
\mu(d)
\\
&\qquad=0
\pmod{(q-\alpha)^2}
\end{aligned}
\]
since $
 \sum_{d\mid\operatorname{rad}(e)}\mu(d)=0$ for $e>1$.  Thus the sum in \eqref{eq:mobius-div} vanishes to second order at every root of exact squared order $e\mid D$, $e>1$.

We now translate the second-order vanishing in $q$ into polynomial
divisibility in $v=(q-q^{-1})^2.$  For $e>2$, the map $q\mapsto v$ is unramified and therefore the corresponding root in the $v$-variable occurs with multiplicity two. For $e=2$, $v+4=(q+q^{-1})^2$, so second-order vanishing in $q-\alpha$ gives one factor $v+4$.  These are exactly the multiplicities in the monic polynomial $U_D(v)$, we obtain \eqref{eq:mobius-div}.

By formula \eqref{eq:intro-key-normalization}, we have
\[
 vC^{\vec\tau}_{\vec\lambda}\in\mathbb Q[v,a^{\pm1}].
\]
Hence
\[
vP_{\vec A}^{\vec\tau}
=
\sum_{\vec\lambda}
\chi_{\vec A}(\vec\lambda)
vC_{\vec\lambda}^{\vec\tau}
\in
\mathbb Q[v,a^{\pm1}].
\]
\end{proof}

\begin{lemma} \label{lem:Mobius-term-reduction}
Let $\vec\lambda$ be a nonempty partition tuple, put $D=D_{\vec\lambda}$ and $r=\ell(\vec\lambda)$, and let $e>1$ divide $D$. Let $\alpha$ satisfy $\ord(\alpha^2)=e$ and set
\[
k=d\ell,
\qquad
d\mid\operatorname{rad}(e),
\qquad
(\ell,e)=1,\qquad \ell\mid D.
\]
Then
\begin{equation}
\label{eq:Mobius-term-local-comparison}
\Psi_{d\ell}
H_{\vec\lambda/(d\ell)}^{\vec\tau}
\equiv
\left(\frac{e}{d}\right)^{r-1}
\Psi_{e\ell}
H_{\vec\lambda/(e\ell)}^{\vec\tau}
\pmod{(q-\alpha)^2}.
\end{equation}
Note that this formula holds tautologically for $e/d=1$.
\end{lemma}
\begin{proof}
Since $
\operatorname{ord}(\alpha^2)=e
$
and
$
(\ell,e)=1,
$
we have
\[
\begin{aligned}
\operatorname{ord}
\left(
(\alpha^{d\ell})^2
\right)
&=
\operatorname{ord}
\left(
(\alpha^2)^{d\ell}
\right)
\\
&=
\frac{e}{\gcd(e,d\ell)}
\\
&=
\frac{e}{d}.
\end{aligned}
\]

For the case
$
\frac{e}{d}>1, 
$
we set $
m:=\frac{e}{d}$
and
$
\vec\nu
:=
\frac{\vec\lambda}{e\ell}.
$
This is well defined since
$
e\mid D$, $\ell\mid D$ and $
(e,\ell)=1
$
imply that
$
e\ell\mid D.
$
Moreover, we have
$
m\vec\nu
=
\frac{e}{d}
\frac{\vec\lambda}{e\ell}
=
\frac{\vec\lambda}{d\ell}.
$

Applying Theorem~\ref{thm:framed-congruence} with variables $(t,b)$, root
$
\beta=\alpha^{d\ell},
$
cycle tuple $\vec\nu$, and scaling factor $m=e/d$, using
\[
\operatorname{ord}(\beta^2)=m,
\]
we obtain
\[
H_{\vec\lambda/(d\ell)}^{\vec\tau}(t,b)
-
\left(\frac{e}{d}\right)^{r-1}
H_{\vec\lambda/(e\ell)}^{\vec\tau}
\left(
t^{e/d},b^{e/d}
\right)
\in
(t-\beta)^2
\mathcal O_\beta[b^{\pm1}].
\]

Now substitute
$
t=q^{d\ell},\qquad b=a^{d\ell}.
$
Then
$
t-\beta
=
q^{d\ell}-\alpha^{d\ell}
=
d\ell\,\alpha^{d\ell-1}(q-\alpha)
+
O((q-\alpha)^2).
$
Since the coefficient
$
d\ell\,\alpha^{d\ell-1}
$
is nonzero. Therefore,
$
(t-\beta)^2
$
is a unit multiple of
$
(q-\alpha)^2
$
in the local ring at $q=\alpha$. Hence, the substituted identity is exactly
\[
\Psi_{d\ell}
H_{\vec\lambda/(d\ell)}^{\vec\tau}
\equiv
\left(\frac{e}{d}\right)^{r-1}
\Psi_{e\ell}
H_{\vec\lambda/(e\ell)}^{\vec\tau}
\pmod{(q-\alpha)^2}.
\]

For the case $e/d=1$, both sides are identical, so the same formula holds
tautologically.
\end{proof}

\section{Arithmetic integrality and completion of the proof}\label{sec:arithmetic}

Section~\ref{sec:rational} proves that
$
vP_{\vec A}^{\vec\tau}
\in
\mathbb Q[v,a^{\pm1}].
$
Thus all nontrivial cyclotomic denominators have already been removed.
The remaining problem is to show that the rational
coefficients are in fact integers. We employ some arithmetic tools to accomplish this.  Note that the argument of this section is independent of  Sections~\ref{sec:hecke},\ref{sec:exterior}, \ref{sec:cable}.  

\subsection{An quantum-integer localization ring}
Put
\begin{equation}\label{eq:R0}
 R_0=\Z[q^{\pm1},a^{\pm1},\{m\}^{-1}:m\ge1]
 =\Z[q^{\pm1},a^{\pm1},\{m\}^{-2}:m\ge1].
\end{equation}
The equality follows from $\{m\}^{-1}=\{m\}\{m\}^{-2}$.  By the definition of the colored HOMFLY--PT invariants, it is easy to see that 
\begin{equation}
\label{eq:framed-colored-invariants-R0}
W_{\vec A}^{\vec\tau}(\mathcal L;q,a)\in R_0.
\end{equation}

Moreover, we have 
\begin{lemma}\label{lem:invariant-localization}
Let $G$ be the group generated by the involutions $q\mapsto -q$ and $q\mapsto q^{-1}$, both fixing $a$.  Then
\begin{equation}\label{eq:invariant-localization}
 R_0^G
 =\Z[v,a^{\pm1},\{m\}^{-2}:m\ge1]
 =\Z[v,a^{\pm1},(vU_m(v))^{-1}:m\ge1].
\end{equation}
\end{lemma}
\begin{proof}
Put $A=\Z[q^{\pm1},a^{\pm1}]$ and let $S$ be the multiplicative set generated by the invariant elements $\{m\}^2$.  Then $R_0=S^{-1}A$.  Invariance under $q\mapsto -q$ first gives $\Z[q^{\pm2},a^{\pm1}]$, and invariance there under $q\mapsto q^{-1}$ gives
\[
 A^G=\Z[q^2+q^{-2},a^{\pm1}]=\Z[v,a^{\pm1}],
\]
since $q^2+q^{-2}=v+2$.  Because every element of $S$ is fixed by $G$, an invariant fraction $f/s$ with $s\in S$ has invariant numerator: $g(f)/s=f/s$ implies $g(f)=f$ in the domain $A$.  Hence
\[
 (S^{-1}A)^G=S^{-1}(A^G).
\]
Finally $\{m\}^2=vU_m(v)$, which gives \eqref{eq:invariant-localization}.
\end{proof}

\subsection{A Frobenius lift in the \texorpdfstring{$p$}{p}-local completion}
\begin{lemma}\label{lem:frobenius-lift}
Fix a prime $p$ and put $R_{0,(p)}=R_0\otimes_\Z\Z_{(p)}$.
Let $\Lambda$ be the integral symmetric-function ring, and let
$(\Lambda^{\otimes L})_d$ denote its part of total degree $d$.
Use the graded completion
\[
 A_p=\prod_{d\ge0}
          \bigl(R_{0,(p)}\otimes_\Z(\Lambda^{\otimes L})_d\bigr).
\]
Then $A_p$ is torsion free and complete for the filtration by total
degree. The Adams operation
\[
 \Psi_p(q)=q^p,\qquad
 \Psi_p(a)=a^p,\qquad
 \Psi_p\bigl(p_m(x^c)\bigr)=p_{pm}(x^c)
\]
extends continuously to $A_p$ and satisfies
\begin{equation}\label{eq:frobenius-lift}
 \Psi_p(F)\equiv F^p\pmod{pA_p}\qquad(F\in A_p).
\end{equation}
The same statement holds after the divided-alphabet substitution
$x\mapsto Y$.
\end{lemma}
\begin{proof}
Each homogeneous part of $\Lambda^{\otimes L}$ is a finite free
abelian group, and $R_{0,(p)}$ is torsion free. The displayed product
is therefore torsion free and complete.

It is enough to check the Frobenius congruence on integral algebra
generators before passing to the graded completion.
It is immediate for $q^{\pm1}$ and $a^{\pm1}$. Since
\[
 \Psi_p(\{m\})=\{mp\}\equiv\{m\}^p\pmod p
\]
and every $\{m\}$ is a unit in $R_0$, inversion gives
\[
 \Psi_p(\{m\}^{-1})\equiv\{m\}^{-p}\pmod p.
\]
In the integral symmetric-function ring, the freshman's-dream
identity in the underlying variables gives
\[
 p_{pm}(x^c)\equiv p_m(x^c)^p\pmod p.
\]
To check integral generators, apply the same identity to the
complete symmetric functions $h_j$, which generate $\Lambda$
over $\Z$:
\[
 \Psi_p(h_j(x^c))
   =h_j\bigl((x_1^c)^p,(x_2^c)^p,\ldots\bigr)
       \equiv h_j(x^c)^p\pmod p.
\]
These identities may first be checked in finitely many variables
and then passed to the symmetric-function ring.
Multiplicativity proves the congruence on
$R_{0,(p)}\otimes\Lambda^{\otimes L}$. Since Adams multiplies
degree by $p$, continuity gives \eqref{eq:frobenius-lift} on $A_p$.

 Finally, each Schur coefficient after $x\mapsto Y$ has only quantum-integer denominators, hence belongs to $R_0$; therefore the substitution defines a continuous endomorphism into the same completed ring.
\end{proof}

\subsection{A Dwork-type lemma}
\begin{lemma}\label{lem:dwork}
Fix a prime $p$. Let $B$ be a commutative torsion-free $\Z_{(p)}$-algebra, and
let $\varphi_p:B\to B$ be an endomorphism satisfying
$\varphi_p(b)\equiv b^p\pmod{pB}$ for every $b\in B$.
Put $A=B[[u_1,\ldots,u_L]]$ and $I=(u_1,\ldots,u_L)$.
Extend $\varphi_p$ continuously to $A$ by
\[
 \psi_p\left(\sum_{\mathbf d}b_{\mathbf d}\mathbf u^{\mathbf d}\right)
       =\sum_{\mathbf d}\varphi_p(b_{\mathbf d})
                                      \mathbf u^{p\mathbf d}.
\]
For $Z\in1+I$, put
\[
 D_p(Z)=\log Z-\frac1p\psi_p(\log Z).
\]
Then $D_p(Z)\in A$. If $\{\varphi_k\}_{k\ge1}$ is a compatible
system of Adams operations on $B$, extended to $\psi_k$ by
$u_c\mapsto u_c^k$, then
\begin{equation}\label{eq:dwork-pairing}
 \sum_{k\ge1}\frac{\mu(k)}k\psi_k(\log Z)
  =\sum_{\substack{k\ge1\\(k,p)=1}}\frac{\mu(k)}k
                                      \psi_k\bigl(D_p(Z)\bigr).
\end{equation}
The sums are interpreted coefficientwise in total $\mathbf u$-degree.
\end{lemma}
\begin{proof}
The Frobenius congruence gives $\psi_p(Z)\equiv Z^p\pmod{pA}$.
Both terms are units with constant term one. Consequently, there
is $K\in I$ with
\[
 \frac{Z^p}{\psi_p(Z)}=1+pK.
\]
Here each positive-degree coefficient of the quotient minus one
is divisible by $p$, and its constant coefficient is zero; this
gives $K\in I$ coefficientwise.
Using continuity and $\psi_p(\log Z)=\log\psi_p(Z)$,
\[
 D_p(Z)=\frac1p\log(1+pK)
       =\sum_{j\ge1}(-1)^{j-1}\frac{p^{j-1}}jK^j.
\]
Since $\ord_p(j)\le j-1$, every coefficient belongs to
$\Z_{(p)}$. The series converges $I$-adically because $K\in I$.
Thus $D_p(Z)\in A$.

For \eqref{eq:dwork-pairing}, terms with $p^2\mid k$ vanish
because $\mu(k)=0$. Every nonzero M\"obius term whose index is
divisible by $p$ is therefore paired uniquely with the term of
index $k$, where $(k,p)=1$. Since $\mu(pk)=-\mu(k)$ and
$\psi_{pk}=\psi_k\psi_p$,
\begin{align*}
 &\frac{\mu(k)}k\psi_k(\log Z)
       +\frac{\mu(pk)}{pk}\psi_{pk}(\log Z)\\
 &\qquad=
 \frac{\mu(k)}k\psi_k
             \left(\log Z-\frac1p\psi_p(\log Z)\right).
\end{align*}
Summing over $(k,p)=1$ gives the identity. This rearrangement is
legitimate: $\psi_k$ multiplies positive total degree by $k$, so
only finitely many terms contribute to a fixed multidegree.
\end{proof}

\subsection{The two HOMFLY--PT involutions}
\begin{lemma}\label{lem:primitive-symmetry}
Let $Y_q$ denote the divided alphabet, $p_m(Y_q)=p_m(y)/\{m\}_q$.  For every integer framing vector $\vec\tau$,
\begin{equation}\label{eq:partition-symmetry}
 Z^{\vec\tau}[Y_{q}]|_{q\rightarrow -q}=Z^{\vec\tau}[Y_q],
 \qquad
 Z^{\vec\tau}[Y_{q}]|_{q\rightarrow q^{-1}}=Z^{\vec\tau}[Y_q].
\end{equation}
Consequently
\begin{equation}\label{eq:plog-symmetry}
 \operatorname{PLog}Z^{\vec\tau}[Y_q]|_{q\rightarrow -q}
 =\operatorname{PLog}Z^{\vec\tau}[Y_q],
 \qquad
 \operatorname{PLog}Z^{\vec\tau}[Y_q]|_{q\rightarrow q^{-1}}
 =\operatorname{PLog}Z^{\vec\tau}[Y_q],
\end{equation}
and every primitive Schur coefficient $P_{\vec A}^{\vec\tau}$ is fixed by both involutions.
\end{lemma}
\begin{proof}
The standard HOMFLY--PT skein involutions (see, e.g., \cite{LP}), componentwise in the colors, give
\begin{equation}\label{eq:colored-q-symmetry}
 W_{\vec A}^{\vec\tau}(-q,a)
 =(-1)^{|\vec A|}W_{\vec A}^{\vec\tau}(q,a),
 \qquad
 W_{\vec A}^{\vec\tau}(q^{-1},a)
 =(-1)^{|\vec A|}W_{\vec A^t}^{\vec\tau}(q,a).
\end{equation}
The signed $q$-framing factors preserve these formulas because every $\kappa_A$ is even and $\kappa_{A^t}=-\kappa_A$.  On the divided alphabet,
\[
 \{m\}_{-q}=(-1)^m\{m\}_q,
 \qquad
 \{m\}_{q^{-1}}=-\{m\}_q.
\]
Hence 
$p_m(Y_{-q})
=
(-1)^m p_m(Y_q), \ p_m(Y_{q^{-1}})
=
-p_m(Y_q).$
Consequently, we have
\begin{equation}\label{eq:schur-q-symmetry}
s_{\vec A}[Y_{-q}]
=
(-1)^{|\vec A|}s_{\vec A}[Y_q], \ s_{\vec A}[Y_{q^{-1}}]
=
(-1)^{|\vec A|}
s_{\vec A^t}[Y_q].
\end{equation}
Then Lemma \ref{lem:primitive-symmetry} follows from \eqref{eq:colored-q-symmetry} and \eqref{eq:schur-q-symmetry}. 
\end{proof}

\subsection{Arithmetic localization of the Schur coefficient invariants}
\begin{proposition}\label{prop:localization}
For every oriented link, integer framing vector $\vec\tau$, and
tuple $\vec A$, the signed Schur coefficient invariants satisfies
\begin{equation}\label{eq:arith-localization}
 P_{\vec A}^{\vec\tau}
       \in\Z[v,a^{\pm1},\{m\}^{-2}:m\ge1].
\end{equation}
\end{proposition}
\begin{proof}
The only possible scalar denominators in the framing-independent
colored invariants $W_{\vec A}$ are quantum integers. One may
compute $W_{\vec A}$ in the preferred zero-framing realization,
where this follows from the projector construction above.
The signed framing multiplier is $(-1)^Eq^K$, with
\[
 E=\sum_b\tau_b|A^b|,\qquad K=\sum_b\tau_b\kappa_{A^b}.
\]
It is a Laurent-monomial unit, so the same localization contains
the framed invariants. It remains to show that the logarithm and
M\"obius combination defining the primitive coefficients leave
no integer denominators in the Schur basis.

Introduce independent degree variables $u_c$ for the link
components. The divided alphabet has the plethystic expression
\[
 Y=\frac{q^{-1}y}{1-q^{-2}}.
\]
The integral Kronecker rule and principal specialization of Schur
functions \cite{Macdonald} show that the Schur coefficients of
$s_A[Y]$ lie in $R_0$. Their hook denominators are products of
$1-q^{-2h}=q^{-h}\{h\}$, as in the proof of
Lemma~\ref{lem:frobenius-lift}.
Thus the degree-recorded partition function is
\[
 Z^{\vec\tau}[u_1Y^1,\ldots,u_LY^L]
   =\sum_{\vec A}W_{\vec A}^{\vec\tau}
                  \prod_c s_{A^c}[Y^c]u_c^{|A^c|}
 \in1+(\mathbf u)
     (R_0\otimes_\Z\Lambda^{\otimes L})[[\mathbf u]].
\]
Each coefficient of fixed $\mathbf u$-multidegree is a finite
symmetric-function sum. Its degree is recorded by $\mathbf u$;
the series represents the same coefficient data as the graded
completion in Lemma~\ref{lem:frobenius-lift}.

Fix a prime $p$, and put
$B_p=R_{0,(p)}\otimes_\Z\Lambda^{\otimes L}$.
Work in $B_p[[\mathbf u]]$ with $I=(\mathbf u)$.
The divided partition function lies in $1+I$.
Apply Lemma~\ref{lem:dwork} with $B=B_p$ and the coefficient
Adams operation $\varphi_p=\Psi_p$, extended by
$\psi_p(u_c)=u_c^p$. Lemma~\ref{lem:frobenius-lift} supplies the
Frobenius lift on coefficient variables, inverse quantum integers,
and symmetric-function variables.
Equation \eqref{eq:dwork-pairing} expresses the primitive
logarithm as a sum of $p$-integral Dwork differences indexed by
$(k,p)=1$. In any fixed multidegree, only finitely many terms
contribute, so passage to the corresponding Schur coefficient is
legitimate. Every fixed Schur coefficient is therefore $p$-integral.
Since the divided-alphabet substitution commutes with Adams, these
are the coefficients $P_{\vec A}^{\vec\tau}$ of
\eqref{eq:PLog}, with their degrees recorded by $\mathbf u$.

This holds for every prime. For completeness, the passage to
$\Z$ may be made coefficientwise. For a fixed coefficient $f$,
choose a nonzero integer $b$ with $bf\in R_0$; only finitely
many rational scalar denominators occur in a fixed multidegree.
For each prime $p\mid b$, the membership $f\in R_{0,(p)}$ supplies
an integer $a_p$ coprime to $p$ with $a_pf\in R_0$.
The integers $b$ and these finitely many $a_p$ have greatest common
divisor one. B\'ezout's identity therefore gives $f\in R_0$.

Every fixed primitive Schur coefficient thus belongs to $R_0$.
Lemma~\ref{lem:primitive-symmetry} shows that it is fixed by both
$q\mapsto-q$ and $q\mapsto q^{-1}$.
Lemma~\ref{lem:invariant-localization} consequently places it in
\[
 R_0^G=\Z[v,a^{\pm1},\{m\}^{-2}:m\ge1],
\]
which is \eqref{eq:arith-localization}.
\end{proof}

\subsection{Completion of the framed LMOV theorem}
\begin{proof}[Proof of Theorem~\ref{thm:framed-LMOV}]
Proposition~\ref{prop:rational-polynomiality} gives
\[
 vP^{\vec\tau}_{\vec A}\in\mathbb Q[v,a^{\pm1}].
\]
By Proposition~\ref{prop:localization}, write
\[
 P^{\vec\tau}_{\vec A}=\frac{A(v,a)}{B(v)},
\]
where $A(v,a)\in\mathbb Z[v,a^{\pm1}]$ and $B(v)$ is a product of the monic polynomials $\{m\}^2=vU_m(v)$.  Divide $vA(v,a)$ by $B(v)$ over $\mathbb Z[a^{\pm1}][v]$.  Rational polynomiality forces the remainder to vanish, so
\[
 vP^{\vec\tau}_{\vec A}\in\mathbb Z[v,a^{\pm1}].
\]
This proves Theorem~\ref{thm:framed-LMOV}.  Setting $\vec\tau=0$ recovers the zero-framed LMOV theorem.
\end{proof}

\appendix
\section{Detailed proof of the local Hecke identities}\label{app:local-hecke}

This appendix contains the representation-theoretic and relative-cocenter calculations
underlying Theorem~\ref{thm:relative-hecke}.  The notation
\[
 \HH=\HH_n(R_\alpha),\qquad \GG=\HH_{2n}(R_\alpha),\qquad A=\HH_L\otimes_{R_\alpha}\HH_R,\qquad
 b=b_n,\qquad C=T_d
\]
is exactly that of Section~\ref{sec:hecke}.  We reserve $\Delta_k$ for the size-$k$ matching
core and $\mathcal D_k$ (with factors $\mathcal D_L,\mathcal D_R$) for distinguished
shuffle sets.  The symbol $\eta_j=(n-j,1^j)$ denotes a hook partition.  The canonical symmetrizing trace of the local Hecke algebra is written $t_{\HH}$; the symbol $\tau$ is therefore reserved throughout the paper for framing data.

\subsection{Local Hecke algebra and standard inputs}
Let $\HH_m$ be the balanced type-$A$ Hecke algebra with braid generators $g_i$ and
\begin{equation}\label{eq:hecke-rel}
 g_i^2=1+zg_i,
\end{equation}
where $z=q-q^{-1}$.  Throughout this section the positive standard basis is the unscaled basis
\[
 T_w:=g_{i_1}\cdots g_{i_{\ell(w)}}
\]
for any reduced expression $w=s_{i_1}\cdots s_{i_{\ell(w)}}$; in particular $T_{s_i}=g_i$.  This notation is distinct from the Murphy normalization $\widehat T_i=qg_i$.  With this convention
\[
 \theta(g_i)=-g_i^{-1}=z-g_i,\qquad \theta(T_w)=(-1)^{\ell(w)}T_{w^{-1}}^{-1}.
\]
The canonical symmetrizing trace is
\[
 t_{\HH}(T_w)=\delta_{w,1},
\]
and $T_w,T_{w^{-1}}$ are dual bases.  We use the integral cocenter basis theorem for finite type-$A$ Hecke algebras in the form developed by Geck--Pfeiffer \cite[Chapter 8]{GP}: minimal-length representatives of conjugacy classes give a free integral cocenter basis.  We also use the annular identification of this cocenter with the positive HOMFLY--PT skein \cite{AM,Lukac}.

Fix $n\ge2$ and a root $\alpha$ with $\ord(\alpha^2)=n$, i.e.
$\alpha^{2n}=1$ and $\alpha^{2j}\ne1$ for $0<j<n$.  Work over the local ring
\[
 R_\alpha=\Q(\alpha)[q^{\pm1}]_{(q-\alpha)}
 =\left\{\frac{f(q)}{g(q)}:\ f,g\in\Q(\alpha)[q^{\pm1}],\ g(\alpha)\ne0\right\}.
\]
Then $z$ and every $\{j\}$ with $0<j<n$ are units, while $\{n\}$ has a simple zero.  Put
\[
 \HH=\HH_n(R_\alpha),\qquad \GG=\HH_{2n}(R_\alpha),\qquad A=\HH\otimes_{R_\alpha}\HH\subset \GG,
 \qquad \eps_n=(-1)^{n-1}.
\]
Let $C$ be the positive block exchange between the two packets.  Let $b=b_n$ be the regular Newton element \eqref{eq:newton-coupon}; write $b_L,b_R$ for its placements in the two packets.

For an algebra $B$ acting on a bimodule $M$ write $[B,M]$ for the $R_\alpha$-linear span of $bm-mb$.  It is an $R_\alpha$-linear submodule, not an ideal.  Every multiplication of a relative cocenter class below is justified by a commutation relation.

We use the corrected Aiston--Morton normalization \cite{AM}.  The generic Schur element of the block labeled by $\lambda\vdash n$ has the form
\begin{equation}\label{eq:schur-element}
 s_\lambda=q^{\kappa_\lambda/2}\prod_{\square\in\lambda}\frac{\{h(\square)\}}{z}.
\end{equation}
A proper Young subalgebra all of whose factors have size $<n$ has unit Schur elements over $R_\alpha$ and is therefore split separable.  The degree-$n$ complete and elementary annular skein elements may have a single factor $\{n\}^{-1}$; all lower degrees are regular.

\subsection{Expansion of the block exchange}
Label positions $0,\dots,2n-1$ for the combinatorial proof.  A partial injection
\[
 f:S\hookrightarrow\{0,\dots,n-1\}
\]
defines the involution $p_f$ exchanging $r$ and $n+f(r)$ for $r\in S$.  Let
\[
 R_k=\sum_{|S|=k}T_{p_f},
\]
the sum over all partial injections of size $k$.

\begin{lemma}\label{lem:full-grid}
There is an exact identity over $\Z[z]$
\begin{equation}\label{eq:full-grid}
 C^2=\sum_{k=0}^nz^kR_k.
\end{equation}
Each $R_k$ commutes with $A$, $R_0=1$, and
\begin{equation}\label{eq:top-grid}
 R_n=\Omega_{\HH}C,
 \qquad
 \Omega_{\HH}=\sum_{w\in S_n}T_{w,L}T_{w^{-1},R}.
\end{equation}
\end{lemma}
\begin{proof}
Write the second block exchange as $n$ rows.  Before row $r$ the active window contains unused high labels or low labels selected by previous rows.  The moving low label $r$ encounters descents exactly at unused high slots.  At a descent the Hecke rule permits either a swap or a stationary term with coefficient $z$.  Once stationarity is chosen at a high slot, the moving label becomes that high label; all remaining comparisons in the row are ascents, so no second stationary choice is possible.  Thus a path is exactly a partial injection and has coefficient $z^{|S|}$.  Distinct injections give distinct endpoint permutations.

For a full matching determined by $w\in S_n$, the endpoint is $(w^{-1}\oplus w)d$, where $d$ is the block exchange.  Its length is $n^2+2\ell(w)$, so the standard basis factorization is length-additive and summing over $w$ gives \eqref{eq:top-grid}.  The double-coset decomposition by the number of exchanged positions shows that each $R_k$ commutes with $A$.
\end{proof}

\subsection{Proper parabolic decomposition}
For $0<k<n$ put $m=n-k$ and define
\[
 B_L=\HH_m\otimes \HH_k\subset \HH,\qquad
 B_R=\HH_k\otimes \HH_m\subset \HH,
\]
and the proper Young subalgebra
\[
 B_k=(B_L)_L\otimes(B_R)_R\subset A.
\]
Let $d_k$ exchange the two matched $k$-strand subpackets and let
\[
 \Omega_k:=\sum_{w\in S_k}T^{(k)}_{w,L}T^{(k)}_{w^{-1},R},
 \qquad
 \Delta_k=\Omega_kT_{d_k}.
\]
Here the superscript $(k)$ only records placement in the two selected matched $k$-strand factors.  Thus $\Omega_k$ is the relative Casimir of those two $\HH_k$-factors.  Let $\mathcal D_k$ be the set of minimal shuffles selecting the matched subsets in both packets.

\begin{lemma}\label{lem:shuffle}
All factors in
\begin{equation}\label{eq:shuffle}
 R_k=\sum_{x\in\mathcal D_k}T_x\Delta_kT_{x^{-1}}
\end{equation}
are length-additive, and $\Delta_k$ commutes with $B_k$.
\end{lemma}
\begin{proof}
Choose matched subsets $S=\{s_0<\cdots<s_{k-1}\}$ and $T=\{t_0<\cdots<t_{k-1}\}$.  If $x=x_L\oplus x_R$ is the corresponding pair of minimal shuffles and $w\in S_k$ is the internal matching permutation, then a direct inversion count gives
\[
 \ell(xp_wx^{-1})=2\ell(x)+k^2+2\ell(w).
\]
Hence every displayed product is its positive standard basis element, and all size-$k$ partial matchings occur once.  The stationary $\HH_m$ factors are disjoint from $\Delta_k$.  On a matched factor the Casimir sliding identity $a_L\Omega_k=\Omega_ka_R$ and block transport $a_RT_{d_k}=T_{d_k}a_L$ give $a_L\Delta_k=\Delta_ka_L$; similarly on the other side.
\end{proof}

\subsection{Higman maps}
Let $B$ be a finite free symmetric $R_\alpha$-algebra with dual bases $e_i,e_i^*$ and let $M$ be a $B$-bimodule.  For background on Higman traces and Casimir operators for symmetric algebras, see \cite[Sections~2.3--2.5]{LorenzTokoly}.  Define the Higman map
\begin{equation}\label{eq:higman}
 h_B(y)=\sum_ie_i^*ye_i.
\end{equation}
It kills $[B,M]$ and takes values in $M^B$.

\begin{lemma}\label{lem:higman}
Let \(B\) be a split separable symmetric \(R_\alpha\)-algebra with unit Schur elements. Then the induced Higman map
\[
\bar h_B:M/[B,M]\rightarrow M^B
\]
is an isomorphism.
(See \cite[Sections~2.3--2.5]{LorenzTokoly}.)
If \(M\) is \(R_\alpha\)-free and
\(
h_B(y)\in\{n\}^aM,
\)
then
\[
y\in[B,M]+\{n\}^aM.
\]
\end{lemma}
\begin{proof}
The first statement is the Higman trace theorem. For the second statement, write
\[
h_B(y)=\{n\}^aY.
\]

Since \(h_B(y)\in M^B\) and \(M\) is \(R_\alpha\)-free, the relation
\[
\{n\}^a(bY-Yb)=0
\]
implies \(bY=Yb\) for all \(b\in B\). Hence \(Y\in M^B\). 

Because the induced Higman map
\(
\bar h_B:M/[B,M]\xrightarrow{\sim}M^B
\)
is an isomorphism, there exists
\(
x\in M
\)
such that
\(
h_B(x)=Y.
\)
Therefore:
\(
h_B(y-\{n\}^ax)=0.
\)
Hence:
\(
y-\{n\}^ax\in\ker h_B=[B,M].
\)
Therefore:
\(
y\in[B,M]+\{n\}^aM.
\)

\end{proof}

For a proper parabolic inclusion $B\subset K$ with
compatible traces, define the norm on a $K$-bimodule $M$ by
\begin{equation}\label{eq:relative-norm}
 N_B^K(y)=\sum_{x\in\mathcal D}T_{x^{-1}}yT_x.
\end{equation}
These relative norms are discussed in \cite{JonesCenters,FrancisJones}.
The distinguished factorization $w=xu$ gives the exact
transitivity identity
\begin{equation}\label{eq:higman-transitivity}
  h_BN_B^K= h_K,
\end{equation}
where $h_B$ and $h_K$  act on the same  $K\text{-bimodule }M$,
with $M\text{ regarded as a }B\text{-bimodule for }h_B$.

Indeed,
\[
 \sum_{u\in W(B)}\sum_{x\in\mathcal D}
 T_{u^{-1}}T_{x^{-1}}yT_xT_u
 =\sum_{w\in W(K)}T_{w^{-1}}yT_w.
\]
The calculation uses length additivity only in the two
outer standard-basis factors and is valid on any
$K$-bimodule $M$.

\subsection{Divisibility of the Higman images}
Let $\eta_j=(n-j,1^j)$ be the hook partitions.  Put
\[
 \kappa_j=n(n-1-2j),
 \qquad
 c_j=q^{\kappa_j/2}\{n-j-1\}!\{j\}!.
\]
Here factorials are products of quantum braces $\{m\}$, and $\{0\}!=1$.

\begin{lemma}\label{lem:higman-square}
The actual element $h_{\HH}(b_n)$ belongs to $\{n\}^2\HH$.
\end{lemma}

\begin{proof}
On a generic simple block, the Higman map is
$s_\lambda\operatorname{Tr}_\lambda$. The
Murnaghan--Nakayama rule \cite[Chapter~I, Section~7]{Macdonald}
says that the $n$-cycle character vanishes off hooks and is
$(-1)^j$ on $\eta_j$. Since $b_n$ closes to $\{n\}P_n$,
\[
 \operatorname{Tr}_{\eta_j}(b_n)=\{n\}(-1)^j,
 \qquad
 \operatorname{Tr}_\lambda(b_n)=0\quad
 (\lambda\text{ not a hook}).
\]
Writing $p_j$ for the generic central hook projectors gives
\begin{equation}\label{eq:central-higman}
 z^n h_{\HH}(b_n)
 =\{n\}^2\sum_{j=0}^{n-1}(-1)^jc_jp_j.
\end{equation}

An individual $p_j$ may have a simple pole. The projector
formula
\[
 p_\lambda=s_\lambda^{-1}
       \sum_{w\in S_n}\operatorname{Tr}_\lambda(T_w)T_{w^{-1}}
\]
has integral character coefficients. A hook of size $n$
has exactly one hook length $n$; a nonhook has none.
Thus $\{n\}p_j$ is regular, and the nonhook projectors
are regular. The total hook projector is regular as well.

Let $\omega=\alpha^n\in\{1,-1\}$. Successive ratios give
$c_{j+1}(\alpha)/c_j(\alpha)=-1$. The primitive product
identity gives
\[
 c_0(\alpha)
 =\alpha^{n(n-1)}\prod_{r=1}^{n-1}(1-\alpha^{-2r})
 =\eps_n n.
\]
Here $\alpha^n=-1$ if $n$ is even, while $n-1$ is even if
$n$ is odd. Hence
\begin{equation}\label{eq:hook-coefficient-values}
 c_j(\alpha)=\eps_n n(-1)^j.
\end{equation}
We may therefore write
\[
 \sum_j(-1)^jc_jp_j
 =\eps_n n\sum_jp_j
  +\sum_j\bigl((-1)^jc_j-\eps_n n\bigr)p_j.
\]
The first term is regular. Each coefficient in the second
sum is divisible by $\{n\}$, so each summand is a regular
multiple of $\{n\}p_j$. The bracket in
\eqref{eq:central-higman} is regular, and $z$ is a unit.
This proves the square factor.
\end{proof}

\begin{lemma}\label{lem:transported-higman-square}
Let
\[
 \theta(g_i)=-g_i^{-1}=z-g_i,
 \qquad t_{\HH}^-:=t_{\HH}\circ\theta,
\]
and let $h_{\HH}^-$ be the Higman map defined by the transported symmetrizing trace $t_{\HH}^-$.  Then
\begin{equation}\label{eq:transported-higman-square}
 h_{\HH}^-(b_n)\in\{n\}^2\HH.
\end{equation}
More precisely, over the generic split semisimple field,
\begin{equation}\label{eq:transported-higman-hooks}
 z^n h_{\HH}^-(b_n)
 =\{n\}^2\sum_{j=0}^{n-1}(-1)^j c_j^-p_j,
 \qquad
 c_j^-=q^{-\kappa_j/2}\{n-j-1\}!\{j\}!.
\end{equation}
If $B\subset \HH$ is a proper standard Young subalgebra preserved by $\theta$ and $\mathcal D$ is the usual set of distinguished coset representatives, define the transported norm
\begin{equation}\label{eq:transported-norm}
 N_{B}^{\HH,-}(y)
 :=\sum_{x\in\mathcal D}\theta(T_{x^{-1}})\,y\,\theta(T_x).
\end{equation}
Then
\begin{equation}\label{eq:transported-transitivity}
 h_B^-N_B^{\HH,-}=h_{\HH}^-,
\end{equation}
and consequently
\begin{equation}\label{eq:transported-norm-div}
 N_B^{\HH,-}(b_n)\in[B,\HH]+\{n\}^2\HH.
\end{equation}
\end{lemma}
\begin{proof}
The map $\theta$ is an involutive $R_\alpha$-algebra
automorphism. Thus $t_{\HH}^-$ is a perfect symmetric trace,
and $h_{\HH}^-$ is an $R_\alpha$-linear map $\HH\to \HH$.
Twisting a generic simple Hecke module by $\theta$ transposes
its partition label. The Schur element for $t_{\HH}^-$ on the
$\lambda$-block is consequently
\[
 s_\lambda^-=s_{\lambda^t}
 =q^{-\kappa_\lambda/2}
       \prod_{\square\in\lambda}\frac{\{h(\square)\}}z.
\]
The same $n$-cycle traces used above give
\eqref{eq:transported-higman-hooks}.

To descend this expression to the local lattice, note that
$c_j^-=q^{-\kappa_j}c_j$ and
$\alpha^{-\kappa_j}=\eps_n$. Indeed,
$\kappa_j=n(n-1-2j)$; the second factor is even for odd
$n$, and is odd for even $n$, when $\alpha^n=-1$.
Together with \eqref{eq:hook-coefficient-values}, this gives
\begin{equation}\label{eq:transported-hook-values}
 c_j^-(\alpha)=n(-1)^j,
 \qquad (-1)^jc_j^-(\alpha)=n.
\end{equation}
Let $p_{\mathrm{hook}}=\sum_jp_j$. As proved above,
$p_{\mathrm{hook}}$ and every $\{n\}p_j$ are regular.
The identity
\[
 \sum_j(-1)^jc_j^-p_j
 =np_{\mathrm{hook}}
  +\sum_j\bigl((-1)^jc_j^- -n\bigr)p_j
\]
therefore has a regular right side: every coefficient in
the second sum belongs to $\{n\}R_\alpha$. Since $z$ is
a unit, this proves \eqref{eq:transported-higman-square}.

For the norm statement, the transported coset systems are
dual because
\[
 t_{\HH}^-\bigl(\theta(T_x)\theta(T_{y^{-1}})\bigr)
 =t_{\HH}(T_xT_{y^{-1}}).
\]
Transporting the parabolic transitivity calculation proves
\eqref{eq:transported-transitivity}. The simple factors of a
proper Young subalgebra have size less than $n$, and their
transported Schur elements are units. Lemma~\ref{lem:higman},
applied with the transported trace, now gives
\eqref{eq:transported-norm-div}.
\end{proof}

In the ordinary cocenter $\overline{\HH}=\HH/[\HH,\HH]$, let $Q_j$
be the Schur class of shape $\eta_j$ and put
$f_j=\{n\}Q_j$. Here $Q_j$ is the class of a primitive
idempotent, so its ordinary trace on the corresponding
simple module is $1$.
The hook Jacobi--Trudi identity \cite[Chapter~I]{Macdonald}
gives
\[
 Q_j=\sum_{r=0}^j(-1)^rh_{n-j+r}e_{j-r}.
\]
The last summand is $(-1)^jh_n=(-1)^jQ_0$, the only one
containing a degree-$n$ complete symmetric function. All
remaining complete and elementary factors have degree
less than $n$ and are regular. Multiplying by $\{n\}$ gives
\begin{equation}\label{eq:hook-lattice}
 f_j=(-1)^jf_0+\{n\}r_j,\qquad r_j\in\overline{\HH}.
\end{equation}
In particular, every $f_j$ is regular.

We also record the logarithmic-derivative cancellation.
Put
\[
 L_j=\frac{q\,d\{j\}/dq}{\{j\}}
    =j\frac{q^j+q^{-j}}{q^j-q^{-j}}\qquad(0<j<n).
\]
At $q=\alpha$, the relations
$\alpha^{n-j}=\alpha^n\alpha^{-j}$ and $\alpha^n=\pm1$
give
\[
 L_{n-j}=-\frac{n-j}{j}L_j.
\]
Pairing $j$ with $n-j$, and using
$\kappa_j=n(n-1-2j)$, yields
\begin{equation}\label{eq:derivative-cancel}
 \sum_{j=0}^{n-1}\kappa_j=0,
 \qquad
 \left.\sum_{j=1}^{n-1}(n-j)L_j\right|_{q=\alpha}=0.
\end{equation}

\subsection{Mixed crossings}
\begin{lemma}\label{lem:full-swap}
Let $\overline{\HH}=\HH/[\HH,\HH]$, let $M_{\mathrm{sw}}=ACA$ be the full-swap double-coset bimodule, and let
\[
 \varsigma:\HH\otimes \HH\longrightarrow \HH\otimes \HH,
 \qquad \varsigma(x\otimes y)=y\otimes x.
\]
Then multiplication by $C$ identifies $M_{\mathrm{sw}}$ with the twisted $A$-bimodule ${}_1A_\varsigma$, and the map
\begin{equation}\label{eq:full-swap-map}
 \Phi_{\mathrm{sw}}:
 M_{\mathrm{sw}}/[A,M_{\mathrm{sw}}]\longrightarrow\overline{\HH},
 \qquad [x_Ly_RC]\longmapsto[xy]
\end{equation}
is an $R_\alpha$-module isomorphism.  Under this isomorphism the full-matching component of
\[
 b_Lb_R(C^2-1)-\eps_n n\{n\}^2b_LC
\]
is
\begin{equation}\label{eq:full-swap-explicit}
 \{n\}^2\sum_{j=0}^{n-1}
 \bigl(c_j-\eps_n n(-1)^j\bigr)f_j.
\end{equation}
\end{lemma}
\begin{proof}
The relations $Ca=\varsigma(a)C$ for $a\in A$ give $ACA=AC$ and identify $u\in A$ with $uC\in M_{\mathrm{sw}}$.  Under this identification,
\[
 [A,M_{\mathrm{sw}}]
 \longleftrightarrow
 \left\langle au-u\varsigma(a):a,u\in A\right\rangle_{R_\alpha}.
\]
Thus the relative cocenter is the twisted coinvariant module
\[
 A_\varsigma:=A/\left\langle au-u\varsigma(a)\right\rangle_{R_\alpha}.
\]
Let $m:A=\HH\otimes \HH\to\overline{\HH}$ be $m(x\otimes y)=[xy]$.  For
$a=a_1\otimes a_2$ and $u=u_1\otimes u_2$,
\[
 m(au-u\varsigma(a))
 =[a_1u_1a_2u_2-u_1a_2u_2a_1]=0,
\]
so $m$ descends to $\Phi_{\mathrm{sw}}$.  It is surjective.  Define
\[
 s:\overline{\HH}\longrightarrow A_\varsigma,
 \qquad [h]\longmapsto[h\otimes1].
\]
In $A_\varsigma$ the defining relations give
\[
 [x\otimes y]=[xy\otimes1]=[1\otimes yx]=[yx\otimes1].
\]
Hence $s$ kills commutators and is well defined.  The same identities show
$\Phi_{\mathrm{sw}}\circ s=\mathrm{id}_{\overline{\HH}}$ and
$s\circ\Phi_{\mathrm{sw}}=\mathrm{id}_{A_\varsigma}$.  Therefore \eqref{eq:full-swap-map} is bijective.

By Lemma~\ref{lem:full-grid}, the full-matching part of $C^2-1$ is $z^nR_n=z^n\Omega_{\HH}C$.  Hence, using $R_n=\Omega_{\HH}C$ and the definition of the Higman map,
\begin{align*}
 \Phi_{\mathrm{sw}}\bigl(z^n b_Lb_RR_n\bigr)
 &=z^n\left[\sum_{w\in S_n}bT_wbT_{w^{-1}}\right]\\
 &=z^n[b\,h_{\HH}(b)].
\end{align*}
In the generic cocenter, \eqref{eq:central-higman}  gives
\[
 z^n h_{\HH}(b)=\{n\}^2\sum_{j=0}^{n-1}(-1)^jc_jp_j.
\]
The annular closure of $b$ is
\[
 B_n=\{n\}P_n=\{n\}\sum_{j=0}^{n-1}(-1)^jQ_j,
\]
so the $\eta_j$-block of the cocenter class of $b$ is
\[
 [p_jb]=(-1)^j\{n\}Q_j=(-1)^jf_j.
\]
Consequently
\begin{equation}\label{eq:full-first-contribution}
 \Phi_{\mathrm{sw}}\bigl(z^n b_Lb_RR_n\bigr)
 =\{n\}^2\sum_{j=0}^{n-1}c_jf_j.
\end{equation}
Similarly,
\[
 \Phi_{\mathrm{sw}}(b_LC)=[b]
 =\sum_{j=0}^{n-1}(-1)^jf_j.
\]
Subtracting $\eps_n n\{n\}^2\Phi_{\mathrm{sw}}(b_LC)$ from \eqref{eq:full-first-contribution} proves \eqref{eq:full-swap-explicit}.  The calculation was written using generic central projectors only to separate simple blocks.  Both sides of \eqref{eq:full-swap-explicit} are regular, and the integral type-$A$ cocenter is free over the local ring; hence the equality descends to the regular cocenter lattice.
\end{proof}

\begin{proof}[Proof of the mixed-crossing identity in Theorem~\ref{thm:relative-hecke}]
Use \eqref{eq:full-grid}.  For $0<k<n$, the $k$th double-coset component of $b_Lb_R(C^2-1)$ is $z^kb_Lb_RR_k$.  By the shuffle formula and cyclic sliding in $\GG/[A,\GG]$,
\[
 [b_Lb_RR_k]
 =\bigl[N_{B_L}^\HH(b)_L\,N_{B_R}^\HH(b)_R\,\Delta_k\bigr].
\]
Transitivity and Lemmas \ref{lem:higman}, \ref{lem:higman-square} imply
\[
 N_{B_L}^\HH(b)\equiv\{n\}^2y_L\pmod{[B_L,\HH]},\qquad
 N_{B_R}^\HH(b)\equiv\{n\}^2y_R\pmod{[B_R,\HH]}.
\]
Because $\Delta_k$ commutes with $B_k$, each replacement is legal in the relative quotient.  Hence every proper component is divisible by $\{n\}^4$.

The bottom component is zero.  By Lemma~\ref{lem:full-swap}, the full-swap component is
\begin{equation}\label{eq:mixed-top}
 \{n\}^2\sum_{j=0}^{n-1}a_jf_j,
 \qquad
 a_j=c_j-\eps_n n(-1)^j.
\end{equation}
Each $a_j$ vanishes at $\alpha$, hence $a_j\in\{n\}R_\alpha$.  Moreover
\[
 \sum_j(-1)^ja_j\in\{n\}^2R_\alpha.
\]
Indeed its value is zero by $c_j(\alpha)=\eps_n n(-1)^j$, and its first $q$-derivative is zero because
\[
 \sum_j\frac{q c_j'}{c_j}
 =\frac12\sum_j\kappa_j+2\sum_{j=1}^{n-1}(n-j)L_j=0
\]
by \eqref{eq:derivative-cancel}.  Substituting \eqref{eq:hook-lattice} gives the explicit decomposition
\[
 \{n\}^2\sum_ja_jf_j
 =\{n\}^2\left(\sum_j(-1)^ja_j\right)f_0
 +\{n\}^3\sum_ja_jr_j,
\]
and both terms lie in $\{n\}^4\overline{\HH}$.  Since $\Phi_{\mathrm{sw}}$ is an $R_\alpha$-module isomorphism, its inverse carries $\{n\}^4\overline{\HH}$ onto $\{n\}^4(M_{\mathrm{sw}}/[A,M_{\mathrm{sw}}])$.  Hence the original full-swap class lies in
\[
 [A,M_{\mathrm{sw}}]+\{n\}^4M_{\mathrm{sw}}
 \subset [A,\GG]+\{n\}^4\GG.
\]
This is precisely the divisibility needed for the top component.
\end{proof}

\subsection{Self crossings}

\begin{lemma}\label{lem:self-full}
After applying $\theta$ to the full-grid identity, the
full-matching contribution to the self remainder, viewed
in $\overline{\HH}\otimes\overline{\HH}$, is, apart from the
factor $\eps_n$,
\begin{equation}\label{eq:self-full-tensor}
 \sum_{j=0}^{n-1}\delta_jf_j\otimes f_j
 -\left(\sum_{j=0}^{n-1}(-1)^jf_j\right)^{\otimes2},
 \qquad
 \delta_j=(-1)^jc_j^-.
\end{equation}
\end{lemma}

\begin{proof}
Since $\theta(C)=(-1)^nC^{-1}$, applying $\theta$ to
\eqref{eq:full-grid} gives
\begin{equation}\label{eq:inverse-grid}
 C-C^{-1}=-\sum_{k=1}^nz^kC\theta(R_k).
\end{equation}
As $R_n=\Omega_{\HH}C$, the $k=n$ term returns to the identity
double-coset summand $A$. The quotient of that summand is
\[
 A/[A,A]\cong\overline{\HH}\otimes_{R_\alpha}\overline{\HH}.
\]
Indeed,
$[A,A]=[\HH,\HH]\otimes \HH+\HH\otimes[\HH,\HH]$, as follows by
expanding commutators of pure tensors and by setting one
factor equal to $1$ for the converse inclusion.

Conjugating the transported Casimir through the block
exchange gives, up to $\eps_n$,
\[
 z^n\sum_{w\in S_n}b\theta(T_{w^{-1}})\otimes\theta(T_w).
\]
We evaluate this tensor over the generic split semisimple
field. On the block of shape $\lambda$, the transported
trace has Schur element $s_\lambda^-$, and its Casimir is
\[
 s_\lambda^-\sum_{u,v}E_{uv}\otimes E_{vu}.
\]
With $b$ inserted in the first factor, projection to the
tensor cocenter gives
\[
 s_\lambda^-\operatorname{Tr}_\lambda(b)
                      Q_\lambda\otimes Q_\lambda.
\]
There is no dimension factor, since $Q_\lambda$ is the
class of a primitive idempotent. The second trace forces
$u=v$, and the first then sums the diagonal entries of $b$.

The $n$-cycle trace vanishes off hooks. On $\eta_j$,
\[
 \operatorname{Tr}_{\eta_j}(b)=\{n\}(-1)^j,
 \qquad
 s_{\eta_j}^-=q^{-\kappa_j/2}
       \frac{\{n\}\{n-j-1\}!\{j\}!}{z^n}.
\]
Multiplication by $z^n$ and $f_j=\{n\}Q_j$ give precisely
$\delta_jf_j\otimes f_j$. The unmixed product has class
\[
 [b]\otimes[b]
 =\left(\sum_j(-1)^jf_j\right)^{\otimes2}.
\]
This proves \eqref{eq:self-full-tensor} generically. Its two
sides are regular, and the ordinary cocenter and its
tensor square are free over $R_\alpha$. The equality
therefore descends to the local tensor lattice.
\end{proof}

We record the support of a proper core. At the permutation
level put $w_k=dd_k$. Since $d_k$ is an involution with
length $k^2$, direct inversion counting gives
\begin{equation}\label{eq:self-core-length}
 \ell(w_k)=\ell(d)-\ell(d_k)=n^2-k^2.
\end{equation}
Thus $d=w_kd_k$ is length-additive, and
$CT_{d_k}^{-1}=T_{w_k}$. Since $\Omega_k\in A$,
\begin{equation}\label{eq:self-proper-core}
 C\theta(\Delta_k)
 =(-1)^{\ell(d_k)}\varsigma(\theta(\Omega_k))T_{w_k}
 \in AT_{w_k}A.
\end{equation}
Let $s_k$ be the minimal representative of the
$W_A$ double coset of $w_k$; see \cite[Sections~2.1.6--2.1.9]{GP}
and \cite[Section~5]{Hahn}. Write
\[
 w_k=r_ks_kt_k,\qquad r_k,t_k\in W_A,
 \qquad
 \ell(w_k)=\ell(r_k)+\ell(s_k)+\ell(t_k).
\]
With
\begin{equation}\label{eq:self-core-factors}
 U_k=(-1)^{\ell(d_k)}\varsigma(\theta(\Omega_k))T_{r_k},
 \qquad V_k=T_{t_k},
\end{equation}
one has
\begin{equation}\label{eq:self-core-factorization}
 C\theta(\Delta_k)=U_kT_{s_k}V_k.
\end{equation}
Because $\theta$ preserves the packet algebras, all outer
shuffle factors lie in $A$. Hence
\begin{equation}\label{eq:self-proper-support}
 C\theta(R_k)\in AT_{s_k}A=:M_k.
\end{equation}
In particular, applying $\theta$ does not move a proper
core into another double-coset bimodule.

The corresponding tensor presentation can be described
explicitly. With $m=n-k$, the minimal representative $s_k$
exchanges the two middle blocks in the ordered block
decomposition $(k,m,m,k)$ and fixes the outer two blocks.
Thus
\[
 J_k:=A\cap T_{s_k}AT_{s_k}^{-1}
       =(B_R)_L\otimes(B_L)_R.
\]
Indeed, minimality makes both $AT_{s_k}$ and $T_{s_k}A$
standard-basis spans, and their intersection is indexed by
$W_A\cap s_kW_As_k^{-1}$. Put
\[
 \varphi_k(j)=T_{s_k}^{-1}jT_{s_k}.
\]
Conjugation exchanges the two middle $\HH_m$ factors of $J_k$
and fixes its outer $\HH_k$ factors. In particular,
$\varphi_k(J_k)=J_k$, and
\[
 jT_{s_k}=T_{s_k}\varphi_k(j)\qquad(j\in J_k).
\]
The length-additive double-coset decomposition gives the
$A$-bimodule presentation
\[
 M_k\cong A\otimes_{J_k,\varphi_k}A,
 \qquad a\otimes c\longmapsto aT_{s_k}c,
\]
where
\[
 aj\otimes c=a\otimes\varphi_k(j)c\qquad(j\in J_k).
\]
The map respects this relation by the displayed conjugation
identity. Distinguished coset representatives give bases on
both sides, and their products with $T_{s_k}$ are
length-additive, proving the isomorphism.

For completeness, the degree-zero Shapiro calculation for
such a twisted tensor product is
\[
 \begin{aligned}
 \frac{A\otimes_{J,\varphi}A}{[A,A\otimes_{J,\varphi}A]}
 &\cong
 \frac{A}{\langle\varphi(j)x-xj:j\in J,\ x\in A\rangle_{R_\alpha}},\\
 [a\otimes c]&\longmapsto[ca].
 \end{aligned}
\]
The balancing relation maps to
$caj-\varphi(j)ca$, and an $A$-commutator maps to zero.
The inverse sends $[x]$ to $[1\otimes x]$: balancing and
cyclicity identify $[1\otimes\varphi(j)x]$ with
$[j\otimes x]=[1\otimes xj]$, while cyclicity also gives
$[a\otimes c]=[1\otimes ca]$.

Write $\mathcal D_k=\mathcal D_L\times\mathcal D_R$,
and define the individual kernels
\begin{equation}\label{eq:self-individual-kernel}
 E_{x_L,k}=\theta(T_{x_L,R})C\theta(\Delta_k)
                                      \theta(T_{x_L^{-1},L}).
\end{equation}
Their factorizations are
\begin{equation}\label{eq:self-individual-factorization}
 \begin{aligned}
 E_{x_L,k}&=A_{x_L}T_{s_k}C_{x_L},\\
 A_{x_L}&=\theta(T_{x_L,R})U_k,\qquad
 C_{x_L}=V_k\theta(T_{x_L^{-1},L}).
 \end{aligned}
\end{equation}
Under the tensor presentation, their transfer has the lift
\[
 \widetilde S_k(h)
 =\sum_{x_L\in\mathcal D_L}h_LA_{x_L}\otimes C_{x_L}
       \in A\otimes_{J_k,\varphi_k}A.
\]
Multiplication sends this lift to
$\sum_{x_L}h_LE_{x_L,k}$. Its relative-cocenter class is
the map $S_k$ defined below. The well-definedness on
$\HH/[B_R,\HH]$ will follow from relative Frobenius duality.

Block transport and the commutation of $\Delta_k$ with
$(B_R)_R$ give the pointwise relation
\begin{equation}\label{eq:self-pointwise-intertwining}
 h_LE_{x_L,k}=E_{x_L,k}h_R\qquad(h\in B_R).
\end{equation}
The same relation holds after summing the kernels over
$\mathcal D_L$.

\begin{lemma}\label{lem:self-proper}
For $0<k<n$, the formula
\begin{equation}\label{eq:self-transfer-map}
 \begin{aligned}
 S_k&:\HH/[B_R,\HH]\longrightarrow M_k/[A,M_k],\\
 S_k([h])&=\sum_{x_L\in\mathcal D_L}[h_LE_{x_L,k}]
 \end{aligned}
\end{equation}
defines an $R_\alpha$-linear map. The analogous right-handed
map, with $B_L$ and $\mathcal D_R$, is also well defined.
Moreover, if
\begin{equation}\label{eq:self-transported-norm}
 N_{B_k}^{A,-}(y)=\sum_{x\in\mathcal D_k}
                 \theta(T_{x^{-1}})y\theta(T_x),
\end{equation}
then
\begin{equation}\label{eq:self-norm-square}
 N_{B_k}^{A,-}(b_LC)\in[B_k,\GG]+\{n\}^2\GG,
\end{equation}
and
\begin{equation}\label{eq:self-proper-divisibility}
 b_LC\theta(R_k)\in[A,\GG]+\{n\}^2\GG.
\end{equation}
\end{lemma}

\begin{proof}
We first prove the assertion about $S_k$. Put
\[
 W_k=C\theta(\Delta_k),\qquad
 \Xi_k=\sum_{x_L\in\mathcal D_L}E_{x_L,k}.
\]
For $B=B_L$, let $\mathsf E_B:\HH\to B$ be the canonical
parabolic projection in the standard basis. Its transported
version is
$\mathsf E_B^-=\theta|_B\circ\mathsf E_B\circ\theta$.
The distinguished right coset systems
\[
 u_x=\theta(T_x),\qquad v_x=\theta(T_{x^{-1}}),
 \qquad x\in\mathcal D_L,
\]
are relative Frobenius dual bases. Explicitly, standard-basis
duality and the decomposition $w=xu$ give, for $h\in \HH$,
\begin{equation}\label{eq:self-relative-duality}
 h=\sum_xu_x\mathsf E_B^-(v_xh)
  =\sum_x\mathsf E_B^-(hu_x)v_x.
\end{equation}
The relative Casimir
\[
 \mathcal E_B^-=\sum_xu_x\otimes_Bv_x
                   \in \HH\otimes_B\HH
\]
therefore satisfies $h\mathcal E_B^-=\mathcal E_B^-h$
for every $h\in \HH$. To see this directly, use the two
expansions in \eqref{eq:self-relative-duality}:
\[
 \begin{aligned}
 h\mathcal E_B^-
 &=\sum_{x,y}u_y\mathsf E_B^-(v_yhu_x)\otimes_Bv_x\\
 &=\sum_yu_y\otimes_B
                \left(\sum_x\mathsf E_B^-(v_yhu_x)v_x\right)\\
 &=\sum_yu_y\otimes_Bv_yh=\mathcal E_B^-h.
 \end{aligned}
\]

For $c\in B_L$, block transport and commutation with
$\theta(\Delta_k)$ give $c_RW_k=W_kc_L$.
The map
\[
 \HH\otimes_{B_L}\HH\longrightarrow M_k,
 \qquad u\otimes v\longmapsto u_RW_kv_L
\]
is consequently balanced, and sends $\mathcal E_{B_L}^-$
to $\Xi_k$. The Casimir identity proves
\begin{equation}\label{eq:self-full-packet-intertwining}
 h_R\Xi_k=\Xi_kh_L\qquad(h\in \HH).
\end{equation}
On the other hand, \eqref{eq:self-pointwise-intertwining} gives
\begin{equation}\label{eq:self-proper-packet-intertwining}
 c_L\Xi_k=\Xi_kc_R\qquad(c\in B_R).
\end{equation}
For $c\in B_R$ and $h\in \HH$, cyclicity in $M_k/[A,M_k]$
and these two identities imply
\[
 \begin{aligned}
 [(c_Lh_L-h_Lc_L)\Xi_k]
 &=[h_L\Xi_kc_L-h_L\Xi_kc_R]\\
 &=[h_Lc_R\Xi_k-h_L\Xi_kc_R]\\
 &=[c_Rh_L\Xi_k-h_L\Xi_kc_R]=0.
 \end{aligned}
\]
The last expression is an $A$-commutator. Thus the map
factors through $\HH/[B_R,\HH]$, proving
\eqref{eq:self-transfer-map}. Interchanging the packets proves
the right-handed statement.

We prove the divisibility assertions on the full
$A$-bimodule $\GG$. For any symmetric trace on $\HH$, let
$e_i,e_i^*$ be dual bases and put
$\Omega=\sum_i e_i\otimes e_i^*$. Casimir duality gives
$\Omega(1\otimes x)=(x\otimes1)\Omega$. Consequently
\begin{equation}\label{eq:self-tensor-higman}
 \begin{aligned}
 \sum_{i,j}e_i^*be_j\otimes e_j^*e_i
 &=\sum_i(e_i^*b\otimes1)\Omega(1\otimes e_i)\\
 &=\sum_i(e_i^*be_i\otimes1)\Omega\\
 &=(h_{\HH}(b)\otimes1)\Omega.
 \end{aligned}
\end{equation}
This is an identity in the actual tensor algebra.

Apply it to the transported trace $t_{\HH}^-$ and write
$\Omega_{\HH}^-=\theta(\Omega_{\HH})$. The tensor-product trace
on $A$ has transported dual bases
$(e_i)_L(e_j)_R$ and $(e_i^*)_L(e_j^*)_R$. Block transport
therefore gives
\begin{equation}\label{eq:self-full-higman}
 \begin{aligned}
 h_A^-(b_LC)
 &=\sum_{i,j}(e_i^*)_L(e_j^*)_R\,b_LC(e_i)_L(e_j)_R\\
 &=\sum_{i,j}(e_i^*be_j)_L(e_j^*e_i)_R\,C\\
 &=(h_{\HH}^-(b)\otimes1)\Omega_{\HH}^-C.
 \end{aligned}
\end{equation}
Lemma~\ref{lem:transported-higman-square} proves that this belongs to
$\{n\}^2\GG$.

The traces on $B_k\subset A$ are compatible. Transporting
the right coset decomposition $w=xu$ gives
\[
 \begin{aligned}
 h_{B_k}^-N_{B_k}^{A,-}(y)
 &=\sum_{u\in W(B_k)}\sum_{x\in\mathcal D_k}
   \theta(T_{u^{-1}})\theta(T_{x^{-1}})
             y\theta(T_x)\theta(T_u)\\
 &=\sum_{w\in W(A)}\theta(T_{w^{-1}})y\theta(T_w)
  =h_A^-(y).
 \end{aligned}
\]
Every factor of $B_k$ has size $k$ or $n-k$, strictly
less than $n$. Its transported Schur elements are therefore
units. The module $\GG$ is free over $R_\alpha$. Applying
Lemma~\ref{lem:higman} to \eqref{eq:self-full-higman}
proves \eqref{eq:self-norm-square}.

Finally, apply $\theta$ to the shuffle formula and move
the trailing shuffle around the relative cocenter:
\begin{equation}\label{eq:self-pairing}
 \begin{aligned}
 [b_LC\theta(R_k)]
 &=\sum_{x\in\mathcal D_k}
   [b_LC\theta(T_x)\theta(\Delta_k)\theta(T_{x^{-1}})]\\
 &=\sum_{x\in\mathcal D_k}
   [\theta(T_{x^{-1}})b_LC\theta(T_x)\theta(\Delta_k)]\\
 &=[N_{B_k}^{A,-}(b_LC)\theta(\Delta_k)].
 \end{aligned}
\end{equation}
The cyclic step is valid because $\theta(T_{x^{-1}})\in A$.
Since $\theta$ preserves $B_k$, the element
$\theta(\Delta_k)$ commutes with $B_k$. Thus, for
$x\in B_k$ and $Y\in \GG$,
\[
 [x,Y]\theta(\Delta_k)=[x,Y\theta(\Delta_k)]\in[A,\GG].
\]
Substituting \eqref{eq:self-norm-square} into
\eqref{eq:self-pairing} proves
\eqref{eq:self-proper-divisibility}.

The map $S_k$ also records the one-packet norm transfer.
Expand the shuffle formula with
$x=x_L\oplus x_R$:
\[
 \begin{aligned}
 [b_LC\theta(R_k)]
 =\sum_{x_L\in\mathcal D_L}\sum_{x_R\in\mathcal D_R}
 \bigl[&b_LC\theta(T_{x_L,L})\theta(T_{x_R,R})\theta(\Delta_k)\\
      &\hspace{1em}\cdot\theta(T_{x_L^{-1},L})
                         \theta(T_{x_R^{-1},R})\bigr].
 \end{aligned}
\]
For each $x_R\in\mathcal D_R$, write
$u=\theta(T_{x_R})$ and $v=\theta(T_{x_R^{-1}})$.
After summing over $x_L$, its contribution is
$[b_Lu_L\Xi_kv_R]$. In the relative quotient,
\[
 \begin{aligned}
 [b_Lu_L\Xi_kv_R]
 &=[v_Rb_Lu_L\Xi_k]
  =[b_Lu_Lv_R\Xi_k]\\
 &=[b_Lu_L\Xi_kv_L]
  =[v_Lb_Lu_L\Xi_k],
 \end{aligned}
\]
where the middle step uses
\eqref{eq:self-full-packet-intertwining} for $v\in \HH$.
Summing over $x_R$ gives
\begin{equation}\label{eq:self-packet-norm}
 [b_LC\theta(R_k)]
 =\sum_{x_L\in\mathcal D_L}
       \bigl[(N_{B_R}^{\HH,-}(b))_LE_{x_L,k}\bigr]
\end{equation}
and hence
\begin{equation}\label{eq:self-packet-transfer}
 [b_LC\theta(R_k)]
 =S_k\bigl([N_{B_R}^{\HH,-}(b)]\bigr)
       \quad\text{in }M_k/[A,M_k].
\end{equation}
By \eqref{eq:transported-norm-div},
$[N_{B_R}^{\HH,-}(b)]=\{n\}^2[y_R]$ in $\HH/[B_R,\HH]$.
Thus the same proper component is $\{n\}^2S_k([y_R])$.
\end{proof}

\begin{proof}[Proof of the self identity in Theorem~\ref{thm:relative-hecke}]
By \eqref{eq:inverse-grid} and Lemma~\ref{lem:self-proper},
every proper component of $b_L(C-C^{-1})$ belongs to
$[A,\GG]+\{n\}^2\GG$. The remaining full-matching term returns
to $A$: indeed,
\[
 C\theta(R_n)=(-1)^n\Omega_{\HH}^-,
\]
because packet exchange preserves the symmetric tensor
$\Omega_{\HH}^-$. Hence the remaining self remainder is
represented by
\[
 \eps_n\bigl(z^nb_L\Omega_{\HH}^- -b_Lb_R\bigr).
\]
By Lemma~\ref{lem:self-full}, its class, apart from
$\eps_n$, is
\begin{equation}\label{eq:self-remaining}
 \sum_{j=0}^{n-1}\delta_jf_j\otimes f_j
 -\left(\sum_{j=0}^{n-1}(-1)^jf_j\right)^{\otimes2}.
\end{equation}

At the root, $\delta_j(\alpha)=n$ for every $j$.
One may see this either from \eqref{eq:transported-hook-values}
or from
\[
 \frac{\delta_{j+1}}{\delta_j}
 =-q^n\frac{\{j+1\}}{\{n-j-1\}},\qquad
 \delta_0(\alpha)
 =\prod_{r=1}^{n-1}(1-\alpha^{-2r})=n.
\]
Furthermore,
\begin{equation}\label{eq:self-delta-sum}
 \sum_j\delta_j-n^2\in\{n\}^2R_\alpha.
\end{equation}
The value at $\alpha$ is zero. The logarithmic derivatives
satisfy
\[
 \sum_j\frac{q\delta_j'}{\delta_j}
 =-\frac12\sum_j\kappa_j
   +2\sum_{r=1}^{n-1}(n-r)L_r,
\]
which vanishes at $\alpha$ by \eqref{eq:derivative-cancel}.
Since every $\delta_j(\alpha)=n$, the ordinary derivative
of $\sum_j\delta_j-n^2$ also vanishes, proving
\eqref{eq:self-delta-sum}.

Insert $f_j=(-1)^jf_0+\{n\}r_j$ into
\eqref{eq:self-remaining} and put
$r_* =\sum_j(-1)^jr_j$. The result is
\[
 \begin{aligned}
 &\left(\sum_j\delta_j-n^2\right)f_0\otimes f_0\\
 &\quad+\{n\}\sum_j(-1)^j(\delta_j-n)
                      (f_0\otimes r_j+r_j\otimes f_0)\\
 &\quad+\{n\}^2
       \left(\sum_j\delta_jr_j\otimes r_j-r_*\otimes r_*\right).
 \end{aligned}
\]
The first coefficient belongs to $\{n\}^2R_\alpha$,
and $\delta_j-n\in\{n\}R_\alpha$. Every term is therefore
divisible by $\{n\}^2$ in
$\overline{\HH}\otimes\overline{\HH}$, and hence in its image
in $\GG/[A,\GG]$. This proves the self identity.
\end{proof}

\begin{remark}\label{rem:small-checks}
For $n=2$, the hooks are $(2)$ and $(1,1)$, with
\[
 \kappa_0=2,\qquad \kappa_1=-2,\qquad
 c_0=q\{1\},\qquad c_1=q^{-1}\{1\}.
\]
At an exact squared-order-two root,
\[
 c_j(\alpha)=-2(-1)^j=\eps_2\,2(-1)^j,
 \qquad
 \left.q\frac{d}{dq}(c_0-c_1)\right|_{q=\alpha}=0.
\]
The self coefficients satisfy
$\delta_0(\alpha)=\delta_1(\alpha)=2$.

For $n=3$, the three hooks have
\[
 (\kappa_0,\kappa_1,\kappa_2)=(6,0,-6),
\]
and
\[
 (c_0,c_1,c_2)
 =\bigl(q^3\{2\}\{1\},\ \{1\}^2,\ q^{-3}\{2\}\{1\}\bigr).
\]
At an exact squared-order-three root,
$c_j(\alpha)=3(-1)^j$ and $\delta_j(\alpha)=3$.
Direct substitution gives
\[
 \sum_j(-1)^ja_j\in\{3\}^2R_\alpha,
 \qquad
 \sum_j\delta_j-9\in\{3\}^2R_\alpha.
\]
These computations check the normalization and signs in
both full-matching formulas.
\end{remark}

\begin{remark}\label{rem:composite}
No primality assumption occurs in Theorem~\ref{thm:relative-hecke}.
The cyclotomic inputs are $\{j\}(\alpha)\ne0$ for
$0<j<n$ when $\ord(\alpha^2)=n$, and the primitive product
identity for $\alpha^2$. The local theorem therefore applies
to composite packet sizes as well. For composite $n$, the
passage from local divisibility by $\{n\}^2$ at an exact
squared-order-$n$ root to a global congruence modulo $U_n(v)$
also requires the roots of proper-divisor squared order.
Those roots are treated separately in the global argument.
\end{remark}

\end{document}